\documentclass[reqno,11pt]{amsart}

\usepackage{latexsym}
\usepackage{geometry}
\usepackage{graphicx}
\usepackage{epstopdf}
\usepackage{a4wide}
\usepackage{amssymb,amsmath}
\usepackage{mathrsfs}
\usepackage{color}
\usepackage{mathabx}
\usepackage{array}

\newcommand{\abs}[1]{\left|#1\right|}

\newcommand{\R}{\mathbb{R}}
\newcommand{\x}{\mathbf{x}}
\newcommand{\vv}{\mathbf{v}}
\newcommand{\uu}{\mathbf{u}}
\newcommand{\ww}{\mathbf{w}}

\newcommand{\Id}{\mathrm{Id}}
\newcommand{\0}{\mathbf{0}}

\renewcommand{\div}{\mbox{\rm div}}

\newtheorem{lemma}{Lemma}
\newtheorem{corollary}{Corrolary}
\newtheorem{definition}{Definition}

\newtheorem{theorem}{Theorem}
\newtheorem{remark}{Remark}

\title{Stability Analysis of Flock and Mill rings for 2nd Order Models in Swarming}

\thanks{DB and JAC were partially supported by the project
MTM2011-27739-C04-02 DGI (Spain) and 2009-SGR-345 from
AGAUR-Generalitat de Catalunya. JAC acknowledges support from the
Royal Society through a Wolfson Research Merit Award. GA
acknowledges support from the Universit\`a di Ferrara through
``Fondi cinque per mille, anno 2009''. JvB acknowledges funding
from NSF grant EFRI-1024765 and NSF grant DMS-0907931. This work
was supported by Engineering and Physical Sciences Research
Council grant number EP/K008404/1.}

\author{G. Albi$^1$, D. Balagu\'e$^2$, J. A.
Carrillo$^3$ and J. von Brecht$^4$}

\address{$^1$Dipartimento di Matematica e Informatica\\
Universit\`a di Ferrara\\Ferrara 44121\\ Italy.\newline E-mail: {\tt
giacomo.albi@unife.it}.}

\address{$^2$Departament de
Matem\`a\-ti\-ques, Universitat Aut\`onoma de Barcelona, E-08193
Bellaterra, Spain. E-mail: {\tt dbalague@mat.uab.cat}.}

\address{$^3$Department of Mathematics\\ Imperial College
London\\ London SW7 2AZ\\ UK.\newline
E-mail: {\tt carrillo@imperial.ac.uk}.}

\address{$^4$Department of
Mathematics, University of California - Los Angeles, Los Angeles,
CA 90095,  USA. E-mail: {\tt jub@math.ucla.edu}.}

\begin{document}
\maketitle

\begin{abstract}
We study the linear stability of flock and mill ring solutions of
two individual based models for biological swarming. The
individuals interact via a nonlocal interaction potential that is
repulsive in the short range and attractive in the long range. We
relate the instability of the flock rings with the instability of
the ring solution of the first order model. We observe that
repulsive-attractive interactions lead to new configurations for
the flock rings such as clustering and fattening formation.
Finally, we numerically explore mill patterns arising from this
kind of interactions together with the asymptotic speed of the
system.
\end{abstract}

\pagestyle{myheadings} \thispagestyle{plain} \markboth{G. ALBI, D.
BALAGU\'E, J.A. CARRILLO, J. VON BRECHT}{STABILITY FOR 2ND ORDER
MODELS IN SWARMING}

\section{Introduction}
Individual-based models (IBMs) appear in biology, mathematics,
physics, and engineering. They describe the motion of a collection
of $N$ individual entities, so the system is defined on a
microscopic scale. IBMs are good models for some applications when
the number of particles is reasonable. Nonetheless, when the
number is large, it is more reasonable to use a continuum model.
Some continuum models, like the one in \cite{Dorsogna2,Dorsogna3},
are derived as a mean-field particle limit leading to a mesoscopic
kinetic description of the problem. At this level, one looks at
the probability density of finding particles at a certain position
and velocity at a given time. Several models have been proposed to
describe the flocking of birds
\cite{Camazine,science1999,ballerini2008interaction,lukeman}, the
formation of ant trails \cite{Dussutour}, the schooling of fish
\cite{Hemelrijk:Hildenbrandt,Bjorn,alethea}{}, swarms of bacteria
\cite{KochWhite}, etc.

These models can include some rules that describe the behavior of
each individual of the system. Such mechanisms can help to
describe the influence of each individual on the others, depending
on their relative position and velocity. For instance, one example
is the classical three zone model \cite{Aoki,Huth:Wissel}. A three
zone model describes how social the individual is in the following
sense. If two individuals are too close, they want to have their
own space (repulsion). When one individual is far from the group,
it wants to go back and socialize (attraction). And finally, in
the group, each individual tries to mimic the behavior of the
others (orientation). Other models just consider rules for
orientation, like the Vicsek model
\cite{Vicsek:Czirok1995,DegondMotsch}. In this case, there is a
mechanism of self-propulsion in which each individual moves with
constant speed and adopts the average direction among their local
neighbors.

We focus our study in the analysis of two particular examples of
IBMs. The first one is a self-propelled interacting particle model
that was introduced in \cite{Levine:2000} and extensively studied in \cite{Dorsogna,Dorsogna2}:
\begin{equation}
\begin{cases}
\displaystyle \dot{x}_j=v_j\\
\displaystyle \dot{v}_j=S(|v_j|)v_j+\frac{1}{N}\sum_{\substack{l=1\\l\neq j}}^N
\nabla W(x_l-x_j)
\end{cases},\quad j=1,\dots, N\label{model}
\end{equation}
We are going to consider the same \emph{self-propulsion/
friction} term used in \cite{Dorsogna,Dorsogna2},
\begin{equation*}
S(|v_j|)=\alpha-\beta|v_j|^2,\qquad \alpha,\beta>0.
\end{equation*}
Note that such a term gives us an asymptotic speed for the
particles, equal to $\sqrt{\alpha/\beta}$. In these references,
the authors study \eqref{model} with pairwise interaction given by
the so-called Morse potential
\[
  U(r)=C_Ae^{-r/l_A}-C_Re^{-r/l_R},
\]
with $C_A$, $C_R$ denoting the attractive and repulsive
strengths and $l_A$, $l_R$ their respective length scales. They
find and describe several patterns for the asymptotic behavior in
2D. They observed flocking behavior, mill on a ring, and
clustering when particles are milling. In \cite{CCR11}, a
well-posedness theory is developed for \eqref{model} proving the
mean-field limit under smoothness assumptions on the potential.
The authors show convergence of the particle model toward a
measure solution of the kinetic equation.

We perform an analysis on the stability of flock rings and the
mill rings as asymptotic solutions for \eqref{model}. The ring
solution was recently studied in \cite{KSUB,BUKB2} where the
authors in this work do a careful general linear analysis of the
rings for the first order model
\begin{equation}\label{eq:firstorder}
\dot{X}_j=\sum_{\substack{l=1\\l\neq j}}^N \nabla
W\left(X_j-X_l\right), \quad j=1,\dots, N.
\end{equation}
Part of the analysis of \eqref{eq:firstorder} is used to study the
the stability of mill rings in \eqref{model}, described in
\cite{Dorsogna}. Related pattern formation in the associated first
order model has been studied in \cite{BUKB,KHP}.

Another second order model that we are going to study is
\begin{equation}
\begin{cases}
\displaystyle \dot{x}_j=v_j\\
\displaystyle \dot{v}_j=\frac{1}{N}\sum_{l=1}^NH(x_j-x_l)(v_l-v_j)+\frac{1}{N}\sum_{\substack{l=1\\l\neq j}}^N
\nabla W(x_l-x_j)
\end{cases},\quad j=1,\dots, N\label{model2}
\end{equation}
with $x_j,v_j\in\mathbb{R}^2$ where the velocity $v_j$ is
described by
the Cucker-Smale alignment term $H$ and the pairwise interaction by a repulsive-attractive radial potential  $W(x)=k(|x|)$.

Even if the analysis has been done in full generality for the
parameter functions of the model $H$ and $W$, we will emphasize
the results in some relevant cases. For instance, we consider the
case of power law repulsive-attractive potentials \cite{KSUB,BCLR}
\begin{equation}
k(r)=\frac{r^a}{a}-\frac{r^b}{b},\qquad a>b>0. \label{potential}
\end{equation}
For the {Cucker-Smale} alignment \cite{CuckerSmale1,CuckerSmale2,Tadmor2008,HaLiu,CFRT}, a relevant case is
$H(x)=g(|x|)$ with
\begin{equation*}
g(r)=\frac{1}{(1+r^2)^\gamma},\qquad \gamma>0.
\end{equation*}

The main results of this work show that the flock ring is unstable
in the second order models \eqref{model} and \eqref{model2} if and
only if its spatial shape is unstable in the first order model
\eqref{eq:firstorder}. We use the same kind of strategy as in
\cite{BUKB2}.

The main idea is to study the stability of the system of ODEs
\eqref{model} by analyzing the eigenvalues of a suitable
linearization with restricted perturbations. We consider
particular perturbations of the flock rings in such a way that
translational invariance is avoided while preserving the mean
velocity. This is really needed since the linearized system
associated to \eqref{model} is always linearly unstable due to
translations. Translational invariance implies the existence of a
generalized eigenvector associated to the zero eigenvalue of the
matrix defining the linearized system. We characterize all cases
in which the linearized system has eigenvalues with zero real part
and their consequences in the instability condition. An analysis
of the stability of the family of flock solutions is under way in
\cite{CHM}.

In addition to flock rings, other spatial shapes are possible as
asymptotic solutions. One can also observe flocks on annuli
(fattening), lines or points (clustering). These patterns can be
explained due to the results in \cite{BCLR2}.

For the mill ring analysis, we start from the results of
\cite{BUKB2} to explore other mill configurations that appear with
repulsive-attractive potentials. We numerically investigate the
formation of fat mills due to the repulsive force and the
formation of clusters when varying the asymptotic speed. In
addition, we show some switching behaviors between flock and mill
rings.

The structure of the paper is as follows. In Section~\ref{sec:MM1}
we study the microscopic and the mesoscopic models and we give the
definitions of the main objects in our studies, the flock and mill
rings. In Section~\ref{sec:linear_flock} we do a linear stability
analysis on the flock rings for models \eqref{model} and
\eqref{model2}. We also explore the fattening and cluster
formation. Finally, in Section~\ref{sec:millstability} we do a
similar study on the stability for mill rings. In all sections we
have performed several numerical tests supporting our theoretical
results.

\section{Microscopic \& Mesoscopic models}\label{sec:MM1}
Let us introduce some particular solutions of the particle model
\eqref{model} and its continuum counterpart.
\subsection{Flock and mill solutions: microscopic model}
\begin{definition}
We call a \emph{flock ring}, the solution of \eqref{model} such
that $\lbrace x_j\rbrace_{j=1}^N$ are equally distributed on a circle with a certain
radius, $R$ and $\lbrace v_j\rbrace_{j=1}^N=u_0$, with $\abs{u_0}=\sqrt{\alpha/\beta}$.
\end{definition}\strut
\begin{definition}
We call a \emph{mill ring}, the solution of \eqref{model} such
that $\lbrace x_j\rbrace_{j=1}^N$ are equally distributed on a
circle with a certain radius, $R$ and $\lbrace v_j\rbrace_{j=1}^N
=u^0_j=\sqrt{\alpha/\beta}\,x_j^\perp/|x_j|$ with $x_j^\perp$ the
orthogonal vector.
\end{definition}

\begin{figure}[ht]
\centering
\includegraphics[scale=0.4]{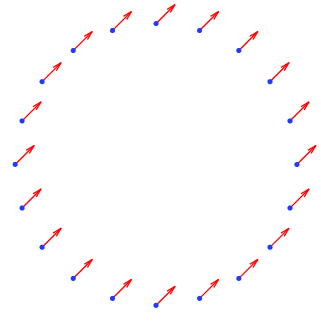}
\hspace{2cm}
\includegraphics[scale=0.4]{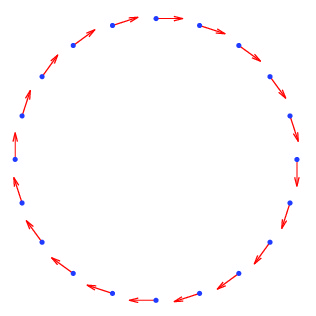}
\caption{Flock and mill ring solutions.}
\end{figure}

By abuse of notation, we will write $\abs{u_0}$ for $\abs{u^0_j}$
since $\abs{u^0_j}=\sqrt{\alpha/\beta}$ for all $j=1,\dots, N$.
Moreover, we will make use of notation $\abs{u_0}$ for both flock
and mill rings indistinctly.

All over the paper, we
will identify $e^{i\theta}\equiv(\cos\theta,\sin\theta)$ and use
$x$ to identify the vector and the complex numbers indistinctly
\begin{equation}\label{eq:flockringsol}
   x_j(t)=R\left(\cos\left(\frac{2\pi}{N}j+\omega t\right),\sin\left(\frac{2\pi}{N}j+\omega t\right)\right)=Re^{i\frac{2\pi j}{N}}e^{i\omega t}.
\end{equation}
In the case of mill rings, we are looking for a solution of the form \eqref{eq:flockringsol}. The case of a flock ring in the comoving frame is equivalent to looking for a solution of the form \eqref{eq:flockringsol} with $\omega=0$. Plugging \eqref{eq:flockringsol} into \eqref{model}, we obtain
\[
    \sum_{\substack{l=1\\l\neq j}}^N\nabla W(x_j-x_l)=0;
\]
therefore, $R$ is determined only by the repulsive-attractive potential.

In the case of mill rings, we have
\[
   v_j(t)=\dot{x}_j(t)=R\omega i e^{i\theta_j}e^{i\omega t},\quad \theta_j=\frac{2\pi j}{N},
\]
and thus, $R^2\omega^2=\frac{\alpha}{\beta}$. Moreover, by taking the derivative
\[
   \dot{v}_j=-\omega^2x_j=-\omega^2\frac{1}{N}\sum_{\substack{l=1\\l\neq j}}^N(x_l-x_j)=0.
\]
Plugging this into \eqref{model}, we get
\[
   \sum_{\substack{l=1\\l\neq j}}^N\left[ \nabla W(x_l-x_j)-\omega^2(x_l-x_j)\right]=\sum_{\substack{l=1\\l\neq j}}^N\nabla \tilde{W}(x_l-x_j)=0,
\]
with $\tilde{W}(x)=W(x)-\omega^2\frac{\abs{x}^2}{2}$. Thus
in order to find the radius for flock and mill rings, we need to
solve the same equation. This expression implies that the
spatial shape has to balance attraction versus repulsion and
centrifugal forces. Now, a direct computation yields
\[
   \abs{x_j-x_l}=2R\sin\left(\frac{(l-j)\pi}{N}\right)
\]
for all times. One can easily compute that
\[
  x_l-x_j=2R\sin\left(\frac{p\pi}{N}\right)
\begin{pmatrix}
-\sin\left(\frac{p\pi}{N}\right) & \cos\left(\frac{p\pi}{N}\right) \\
\\
\cos\left(\frac{p\pi}{N}\right) & \sin\left(\frac{p\pi}{N}\right)
\end{pmatrix}
\begin{pmatrix}
\cos\left(\theta_j\right)\\
\\
\sin\left(\theta_j\right)
\end{pmatrix}, \quad p=l-j,
\]
and
\begin{align*}
\sum_{\substack{l=1\\l\neq j}}^N\nabla\tilde{W}(|x_j-x_l|)&=\sum_{\substack{l=1\\l\neq j}}^N(x_j-x_l)\frac{\tilde{k}'(|x_j-x_l|)}{|x_j-x_l|}\\
&=\sum_{\substack{l=1\\l\neq j}}^N
\begin{pmatrix}
-\sin\left(\frac{p\pi}{N}\right)\cos\left(\theta_j\right) + \cos\left(\frac{p\pi}{N}\right)\sin\left(\theta_j\right)\\
\\
\cos\left(\frac{p\pi}{N}\right)\cos\left(\theta_j\right)+
\sin\left(\frac{p\pi}{N}\right)\sin\left(\theta_j\right)
\end{pmatrix}\tilde{k}'(|x_j-x_l|).
\end{align*}
By symmetry we can assume $j=N$, and thus
\[
\sum_{\substack{l=1\\l\neq j}}^N\nabla\tilde{W}(|x_j-x_l|) = \sum_{p=1}^{N-1}
\begin{pmatrix}
-\sin\left(\frac{p\pi}{N}\right)\\
\\
\cos\left(\frac{p\pi}{N}\right)
\end{pmatrix}\tilde{k}'(|x_j-x_l|),
\]
changing $j$ to $N-j$, we finally obtain
\begin{equation}\label{eq:condition15}
\sum_{\substack{l=1\\l\neq j}}^N\nabla\tilde{W}(|x_j-x_l|)=\begin{pmatrix}
-\displaystyle\sum_{p=1}^{N-1}\sin\left(\frac{p\pi}{N}\right)\tilde{k}'\left(2R\sin\left(\frac{p\pi}{N} \right)\right)\\
\\
0
\end{pmatrix}.
\end{equation}
As a conclusion, the radius of a flock or mill ring solution is characterized by
\[
\displaystyle\sum_{p=1}^{N-1}\sin\left(\frac{p\pi}{N}\right)\tilde{k}'\left(2R\sin\left(\frac{p\pi}{N} \right)\right)=0.
\]
For general potentials there can be more than one flock or mill solution. In the case of the power law potentials \eqref{potential}, there is only one solution.
Condition \eqref{eq:condition15} reads
\begin{equation}\label{eq:radiuseq}
 \left(2R\right)^{a-1}\frac{1}{N}\sum_{p=0}^{N-1}\sin^a\left(\frac{p\pi}{N}\right)-\left(2R\right)^{b-1}\frac{1}{N}\sum_{p=0}^{N-1}\sin^b\left(\frac{p\pi}{N}\right)-2R\omega^2\frac{1}{N}\sum_{p=0}^{N-1}\sin^2\left(\frac{p\pi}{N}\right)=0.
\end{equation}
To prove uniqueness, we notice that the function $f(r)=C_1 r^a - C_2 r^b -C_3$ with $a>b>0$ and $C_1>C_2>0$, $C_3>0$, {has only one zero}.
Computing the first derivative and looking for critical points,
we obtain $r_1=0$ and $r_2^{{a-b}}=\frac{C_2 b}{C_1
 a}$. Taking the second derivative and
evaluating at $r_2$, one obtains $f''(r_2)= r_2^{b-2}C_2b(a-b) >0$,
so $r_2$ is a local minimum. For $0<r<r_2$, one has $f'(r)<0$ whereas for all $r\in (r_2,+\infty)$, one has $f'(r)>0$. Then we conclude that $f(r)$ has a unique zero.
Notice that the solution to \eqref{eq:radiuseq} depends on the number of particles and we will use the notation $R=R(N)$.
\subsection{Flock and mill solutions: mesoscopic model}
In this subsection we characterize the radius of flock and mill rings.  We introduce the function
\begin{equation*}
\psi_\alpha(s)=\frac{1}{\pi}\int_{0}^{\pi}{(1-s\cos\theta)(1+s^2-2s\cos\theta)^\frac{\alpha-2}{2}}\,d\theta,
\end{equation*}
already analyzed in \cite{BCLR}. A change of variables in the previous function shows that
\begin{equation}\label{eq:psifunc}
\psi_\alpha(1)=\frac{2^{\alpha-1}}{\pi} B\left( \frac{\alpha+1}{2},\frac{1}{2}\right).
\end{equation}
\begin{lemma} If $N\to \infty$ then $R(N)\to R_{ab}(\abs{u_0})$ where $R_{ab}(\abs{u_0})$ is the solution of the following equation:
\[
    \psi_a(1)R^{a-1}-\psi_b(1)R^{b-1}-{\omega^2}{R}=0.
\]
\end{lemma}
\begin{proof}
We first take
$2^{a-1}\frac{1}{N}\sum_{p=0}^{N-1}\sin^a(\frac{p\pi}{N})$.
Multiplying and dividing by $\pi$ we obtain the following equality
\begin{equation}\label{eq:intsin^a}
\lim_{N\to\infty}2^{a-1}\frac{1}{\pi}\left(\frac{\pi}{N}\sum_{p=0}^{N-1}\sin^a\left(\frac{p
\pi}{N}\right)\right)=
2^{a-1}\frac{1}{\pi}\int_{0}^{\pi}\sin^a(x)\,dx =
2^{a-1}\frac{1}{\pi} \left(2\int_{0}^{\pi/2}\sin^a(x)\,dx\right).
\end{equation}
Now, we use the following expression for the Beta function $$B(x,y)=2\int_{0}^{\pi/2}(\cos \theta)^{2x-1}(\sin\theta)^{2y-1}\,d\theta,$$
with $x=\frac{1}{2}$, $y=\frac{a+1}{2}$, and using that
$B(x,y)=B(y,x)$ in \eqref{eq:intsin^a} together with \eqref{eq:psifunc} to obtain
\begin{equation*}
\lim_{N\to\infty}2^{a-1}\frac{1}{\pi}\left(\frac{\pi}{N}\sum_{p=0}^{N-1}\sin^a\left(\frac{p
\pi}{N}\right)\right)=
2^{a-1}\frac{1}{\pi}B\left(\frac{a+1}{2},\frac{1}{2}\right)=\psi_a(1).
\end{equation*}
The same reasoning works by changing $a$ for $b$ in the second
term in \eqref{eq:radiuseq}. For the third term we use the fact
that we can compute the exact sum
\[
    \frac{2}{N}\sum_{p=0}^{N-1}\sin^2\left(\frac{p\pi}{N}\right)=1.
\]
\end{proof}
\begin{remark} In the case of flock rings $\omega=0$, their radius is determined by the radius of the aggregation ring found in {\rm\cite{BCLR}}
\[
R(N)\to R_{ab}=\frac{1}{2}\left( \frac{B( \frac{b+1}{2},\frac{1}{2})}{B(
\frac{a+1}{2},\frac{1}{2})}\right)^\frac{1}{a-b}
\quad\text{as}\quad N\to\infty.
\]
\end{remark}
\begin{remark}
Let $W(x)=k(|x|)$ be a general interaction potential. Call $f(r)=-k'(r)/r$. Then the radius of the ring is determined by
\[
   \int_{0}^{\frac{\pi}{2}}f(2R\sin(\theta))\sin^2(\theta)\,d\theta=0,
\]
as shown in {\rm\cite{BUKB2}}.
\end{remark}
\begin{remark}
The corresponding mesoscopic model to the particle system
\eqref{model}, as proven in {\rm\cite{CCR11}}, is given by the
kinetic equation
\begin{equation}\label{eq:kineticmodel}
  \frac{\partial f}{\partial t}+v\cdot\nabla_x f +\div_v[(\alpha-\beta|v|^2)vf)]-\div_v[(\nabla_x W\ast \rho)f]=0,
\end{equation}
where
\[
   \rho(t,x)=\int_{\R^2} f(t,x,v)\,dv.
\]
It was shown in {\rm \cite{Dorsogna3}} that singular solutions of the type
\[
   f(t,x,v)=\rho(t,x)\delta\left(v-u_0\right),\qquad f(t,x,v)=\rho(t,x)\delta\left(v-\sqrt{\frac{\alpha}{\beta}}\frac{x^\perp}{x}\right),
\]
with $\rho(t,x)$ the uniform distribution on a ring, are weak solutions of the kinetic model \eqref{eq:kineticmodel}, called the flock and mill ring continuous solutions respectively.
\end{remark}
\section{Linear stability analysis for flock rings}\label{sec:linear_flock}
We will now focus on the stability analysis of flock rings for some particular perturbations in terms of the parameters of the model $(a,b,u_0)$. We take advantage of the careful stability analysis of the ring solutions of the aggregation equation performed in \cite{BUKB2}.
\subsection{Stability of flock solutions without the Cucker-Smale term}\label{sec:nocuckersmale}
We consider the model (\ref{model})
and we perform the change of variables to the comoving frame
\begin{equation}
\begin{cases}
y_j(t)=x_j(t)-u_0t\\
z_j(t)=v_j(t)-u_0
\end{cases}\qquad j=1,\ldots, N,\label{eq:changeflock}
\end{equation}
where $u_0$ is the asymptotic velocity of a fixed flock ring. Therefore the system
\eqref{model} reads
\begin{equation*}
\begin{cases}
\displaystyle\frac{d}{dt}y_j=v_j-u_0=z_j
\\
\displaystyle\frac{d}{dt}z_j=\underbrace{(\alpha-\beta|z_j+u_0|^2)}_{S_0(|z_j|)}(z_j+u_0)-\frac{1}{N}\sum_{\substack{k=1 \\ k\neq j}}^N
\nabla W(y_j-y_k)
\end{cases}, \quad j=1,\ldots, N. 
\end{equation*}

A {flock ring} can then be characterized as a stationary solution of the form $\left(y_{j}^0,z_{j}^0\right)=\left(Re^{i\theta_j},0\right),
\label{tflocksol}$
where $\theta_j=\frac{2\pi j }{N}$ for $j=1,\dots,N$.
This stationary solution satisfies
\[
S_0(|z^0_j|)=0,\quad \nabla S_0(|z^0_j|)=-2\beta u_0u_0^*,\quad
0=\sum_{k\neq
j}k'(|y^0_j-y^0_k|)\frac{(y^0_k-y^0_j)}{|y^0_k-y^0_j|}.
\]
As in \cite{BUKB2}, we
restrict the set of possible perturbations of the {flock solution} to those of the form
\begin{equation*}
\tilde{y}_j(t)=Re^{i\theta_j}(1+h_j(t)),\qquad
\end{equation*}
where $h_j\in\mathbb{C}$, such that $|h_j|\ll1$ and
\begin{equation}\label{eq:perturbationhj}
    \sum_{j=1}^N h_j(t)=\sum_{j=1}^N h_j'(t) = 0.
\end{equation}
The first restriction is to avoid the zero eigenvalue due to
translations. The second one comes from the fact that the mean
velocity of the perturbed system should be $u_0$. More general
perturbations will generically lead to other flock solutions with
different asymptotic velocity $u_0$. Their orbital stability will
be analyzed elsewhere \cite{CHM}. Therefore, the perturbed system
reads
\begin{equation*}
\begin{cases}
\displaystyle\frac{d}{dt}\tilde{y}_j(t)=Re^{i\theta_j}h'_j=\tilde{z}_j
\\
\displaystyle\frac{d}{dt}\tilde{z}_j(t)=Re^{i\theta_j}h''_j=S_0(|\tilde{z}_j|)-\frac{1}{N}\sum_{\substack{l=1 \\ l\neq j}}^N
\nabla W(\tilde{y}_j-\tilde{y}_l)
\end{cases}, \quad j=1,\ldots, N.
\label{flock4}
\end{equation*}

The linearization of the system around the {flock solution} $(y^0_j,z^0_j)$ reads as 
\begin{equation*}
\begin{cases}
\displaystyle\frac{d}{dt}\tilde{y}_j(t)=Re^{i\theta_j}h'_j=\tilde{z}_j
\\
\displaystyle Re^{i\theta_j}h''_j=-2\beta u_0u_0^*-\frac{1}{N}\sum_{\substack{l=1 \\ l\neq j}}^N
\nabla W(\tilde{y}_j-\tilde{y}_l)
\end{cases},\quad j=1,\ldots, N. \label{flock5}
\end{equation*}
From the previous equation, we can characterize $h_j''$ as
$$h_j''=\sum_{l\neq j}\left[G_1(\phi/2)(h_j-e^{i\phi}h_l)+G_2(\phi/2)(\bar{h_l}-e^{i\phi}\bar{h_j})]\right]-2\beta u_0u_0^Th'_j,$$

where $\phi=\frac{2\pi(l-j)}{N}$ and
\begin{align*}
G_1(\phi)=&\frac{1}{2N}\left[-a(2R|\sin\phi|)^{a-2}+b(2R|\sin\phi|)^{b-2}
\right],
\\
G_2(\phi)=&\frac{1}{2N}\left[-(a-2)(2R|\sin\phi|)^{a-2}+(b-2)(2R|\sin\phi|)^{b-2}
\right],
\end{align*}
for the power law potentials.
The details of these previous computations for general potentials can be found in {\rm\cite{BUKB2}}.

Let us consider the following ansatz for perturbations $h_j$
\begin{equation}\label{eq:ansatz}
h_j=\xi_+(t)e^{im\theta_j}+\xi_{-}(t)e^{-im\theta_j},\qquad m=2,3,\dots,
\end{equation}
which satisfies conditions \eqref{eq:perturbationhj}.
We need to exclude the case $m=1$ since it leads to a zero eigenvalue due to the rotational invariance of the system.
Following the same strategy as in \cite{BUKB2} and some computations, we finally deduce
\[
\begin{pmatrix}
\xi_+''
\\
{\xi_-''}
\end{pmatrix}
= \underbrace{
\begin{pmatrix}
I_1(m) & I_2(m)
\\
I_2(m) & I_1(-m)
\end{pmatrix}
}_{{M}}
\begin{pmatrix}
\xi_+
\\
{\xi_-}
\end{pmatrix}
-2\beta u_0u_0^T
\begin{pmatrix}
\xi'_+
\\
{\xi'_-}
\end{pmatrix},
\]
with $I_1$ and $I_2$ real functions given by
\begin{align}
I_1(m)=& \sum_{l\neq
j}G_1(\phi/2)(1-e^{i(m+1)\phi})=4\sum_{p=1}^{N/2}G_1\left(\frac{\pi
p}{N}\right)\sin^2\left(\frac{(m+1)\pi p}{N}\right),\label{eq:I1}
\\
I_2(m)=& \sum_{l\neq j}G_2\left(\phi/2\right)(e^{i m
\phi}-e^{i\phi})=4\sum_{p=1}^{N/2}G_2\left(\frac{\pi
p}{N}\left)\left[\sin^2\left(\frac{\pi p}{N}\right)-\sin^2\right(\frac{m\pi
p}{N}\right)\right].\label{eq:I2}
\end{align}
The previous system can be written also in the following form
\begin{equation}\label{eq:bigmatrix1}
\frac{d}{dt}
\begin{pmatrix}
\xi_+ \\
\bar{\xi_-}\\
\eta_+ \\
\bar{\eta_-}
\end{pmatrix}
=
\begin{pmatrix}
{0} &\Id\\
{M} & {-2\beta u_0u_0^T}
\end{pmatrix}
\begin{pmatrix}
\xi_+\\
\xi_-\\
\eta_+\\
\eta_-
\end{pmatrix}=L
\begin{pmatrix}
\xi_+\\
\xi_-\\
\eta_+\\
\eta_-
\end{pmatrix},
\end{equation}
where $(\eta_+,{\eta}_-)=(\xi'_+,{\xi}'_-)$.

If we do not assume \eqref{eq:perturbationhj} and
\eqref{eq:ansatz}, then we cannot reduce the analysis to a
$4\times4$ system. An arbitrary perturbation for general flocks
leads instead to a matrix of the form
\[
L=
\begin{pmatrix}
0 & \; \;\mathrm{Id}\\
\\
\mathbf{M} & -2\beta U
\end{pmatrix},
\]
where the partition into $2N \times 2N$ sub-blocks reflects the
distinction between position and velocity contributions to the
Jacobian: The symmetric matrix $\mathbf{M}$ is the $2N\times 2N$
Hessian that results from linearizing the first order system
\eqref{eq:firstorder} about a given flocking configuration,
whereas $U$ denotes a block-diagonal matrix with $N$ blocks of the
$2 \times 2$ matrix ${u_0}{u_0^T}$ along the diagonal. By
rotational invariance we can reduce to the case $u_0=e_1=(1,0)$,
so that the block matrix $U$ acts on $\x = (x_1,\ldots,x_N)^{T}
\in \mathbb{R}^{2N},$ $x_i \in \mathbb{R}^2$, according to the
relation
\begin{equation*}
(U\x)_{i} = \begin{pmatrix} \langle x_i , e_1 \rangle \\ 0 \end{pmatrix}.
\end{equation*}

We now turn to the task of characterizing the eigenvalues of $L$
in terms of the eigenvalues of $\mathbf{M}$. In other words, we
aim to characterize the stability of a flock in terms of the
stability of its spatial shape as a solution to the first order
model. To fix the notation, we write the eigenvalue problem for
the flock as
\begin{equation}
\lambda
\begin{pmatrix}
\x \\ \vv
\end{pmatrix}
=
\begin{pmatrix}
0 & \Id \\
\mathbf{M} & -2\beta U
\end{pmatrix}
\begin{pmatrix}
\x \\ \vv
\end{pmatrix} =
L
\begin{pmatrix}
\x \\ \vv
\end{pmatrix},
\label{eq:eig}
\end{equation}
where the matrix $\mathbf{M}$ determines the stability of the
flocking configuration as a solution of the first order model. For
any given eigenvector $(\x,\vv) \in \mathbb{C}^{2N} \times
\mathbb{C}^{2N}$ of the full system \eqref{eq:eig}, we always
assume the normalization $\x^*\x=1$. Substituting the first
equation $\lambda \x = \vv$ into the second equation yields the
equivalent statement
\begin{equation}
\lambda^2 \x + 2\beta \lambda U \x - \mathbf{M}\x = 0.
\label{eq:eivec}
\end{equation}
Let $|\x|_{2}$ denote the semi-norm on $\mathbb{C}^{2N}$ defined according to
\begin{equation*}
|\x|^{2}_{2} := \sum^{N}_{i=1} |\langle x_i , e_1 \rangle| ^2,
\label{eq:semi}
\end{equation*}
and let $E^{N} \cong \mathbb{C}^{N}$ denote the subspace
\begin{equation*}
E^{N} := \left\{ \x \in \mathbb{C}^{2N} : |\x|_2 = 0 \right\} = \ker(U).
\end{equation*}
Premultiplying by $\x^*$, the fact that $\x^* U \x = |\x|^{2}_2$, the normalization on $\x$ and the
quadratic formula combine to imply the key identity
\begin{equation}
\lambda = -\beta |\x|^2_2 \pm \sqrt{ \beta^2|\x|^4_2 +
\x^*\mathbf{M}\x}. \label{eq:lambda}
\end{equation}

As $\mathbf{M}$ is symmetric, we may write its $2N$ real eigenvalues and corresponding normalized
($\x^*\x=1$) eigenvectors as
\begin{equation*}
\mu_{2N} \leq \mu_{2N-1} \leq \cdots \leq \mu_2 \leq \mu_{1} \qquad
\mathbf{M} \x_{i} = \mu_{i} \x_i.
\end{equation*}
The notation $a_L(\lambda)$, $a_\mathbf{M}(\mu)$ will denote the algebraic
multiplicities of $\lambda,\mu$ as eigenvalues of their respective
matrices. The bulk of the analysis lies in characterizing the
eigenvalues $\lambda$ of the full system \eqref{eq:eig} that have
$\Re(\lambda) = 0$.
\begin{lemma}
\label{lem:nonzeros} Let $\lambda$ denote an eigenvalue of
\eqref{eq:eig}. Then $\Re(\lambda) = 0$ and $\Im(\lambda) \neq 0$
if and only if $\lambda = \pm i \sqrt{-\mu_{k}}$ for some $k$ with $\mu_k
< 0$ and $\x_k \in E^{N}$. The eigenspace
consists only of eigenvectors.
\end{lemma}
\begin{proof}
If $\x_k \in E^N$ then \eqref{eq:eivec} reads
$\lambda^2 \x_k = \mathbf{M} \x_k$, or equivalently $\lambda^2 =
\mu_k.$
To have $\Im(\lambda) \neq 0$ then requires $\mu_k < 0$.
Conversely, if $(\x,\lambda \x)$ denotes an eigenvector with
$\Re(\lambda) = 0$ and $\Im(\lambda) \neq 0$, the formula
\eqref{eq:lambda} implies that necessarily $\x \in E^N$,
and therefore $\mathbf{M}\x = \lambda^2 \x$. Thus $\lambda^2 = \mu_k$ for
some $\mu_k < 0$.

To show the last statement, suppose a generalized eigenvector
existed that is not an eigenvector. Then there exists an
eigenvector $(\x,\lambda \x)$ with $\x \in E^N$ so that
the system of equations
\begin{equation}
\begin{pmatrix}
-\lambda \Id && \Id \\
\mathbf{M} && -2\beta U - \lambda \Id
\end{pmatrix}
\begin{pmatrix}
\uu \\ \ww
\end{pmatrix}
=
\begin{pmatrix}
\x \\ \lambda \x
\end{pmatrix}
\label{eq:gen_eivec}
\end{equation}
has a non-trivial solution. Substituting the first equation $\ww =
\lambda \uu + \x$ into the second equation, then pre-multiplying
by $\x^*$ demonstrates
\begin{align*}
\mathbf{M} \uu - 2\beta U \ww &= 2\lambda \x + \lambda^2\uu \\
\x^*\mathbf{M}\uu &= 2\lambda + \lambda^2 \x^*\uu.
\end{align*}
The last line follows as $\x^*\x = 1$ and $\x \in E^N = \ker(U)$.
The symmetry of $\mathbf{M}$ and the fact that $\mathbf{M}\x = \lambda^2 \x$ combine
to show $\x^*\mathbf{M}\uu = \lambda^2 \x^*\uu$. Thus $\lambda = 0$, leading to a contradiction.
\end{proof}
\begin{lemma}\label{lem:gen_eivecs}
Let $\beta > 0$. Then $\lambda = 0$ is an eigenvalue of
\eqref{eq:eig} and $(\x,\0)$ is a corresponding eigenvector if and
only if $\mathbf{M}\x = \0$. If $\x \in E^N$ then $(\x,\0)$
generates a single generalized eigenvector, whereas if $\x \notin
E^N$ then $(\x,\0)$ generates no generalized
eigenvectors.
\end{lemma}
\begin{proof}
The first statement follows trivially from \eqref{eq:eivec}. To
see the second statement, consider the system of equations
\eqref{eq:gen_eivec} with $\lambda = 0$. This reduces to the
equations $\ww = \x$ and
\begin{equation*}
\mathbf{M}\uu = 2\beta U\x,
\end{equation*}
which by premultiplying by $\x^*$ as before and using the fact
that $\mathbf{M}\x = \0$ necessitates $\x \in E^N$ as $\beta > 0$.
If indeed $\x \in E^N$ then any $\uu \in \ker(\mathbf{M})$
suffices. Without loss of generality, take $\uu = \x$ itself. If
$(\x,\0)$ generates a second generalized eigenvector then the
system of equations

\begin{equation*}
\begin{pmatrix}
0 && \Id \\
\mathbf{M} && -2\beta U
\end{pmatrix}
\begin{pmatrix}
\uu \\ \ww
\end{pmatrix}
=
\begin{pmatrix}
\x \\ \x
\end{pmatrix}
\end{equation*}
has a non-trivial solution. As then $\ww = \x$ and $\x \in
E^N$ this reads $\mathbf{M}\uu = \x$. Premultiplying one last time
by $\x$, the facts that $\mathbf{M}\x = \0$ and $\x^*\x = 1$ combine to
produce the contradiction $0 = 1$.
\end{proof}
This lemma yields, as a corollary, the algebraic multiplicity
$a_L(0)$ of zero as an eigenvalue of the second order system.
\begin{corollary}
Let $\beta > 0$. Then
\begin{equation*}
a_L(0) = \dim( \ker(\mathbf{M}) \cap E^N ) + \dim( \ker(\mathbf{M}) ).
\end{equation*}
\end{corollary}
Let $a_{\mathbf{M},\perp}(0) := \dim( \ker(\mathbf{M}) \cap E^N )$, so that
$a_L(0) = a_{\mathbf{M},\perp}(0) + a_\mathbf{M}(0)$. Note that neither quantity
depends on $\beta$, and the conclusion holds whenever $\beta$ is
positive. Thus, if $\beta \in (0,\infty)$ it follows that
$a_{L}(0)$ is constant. Moreover, Lemma~\ref{lem:nonzeros} holds
uniformly in $\beta$ as well. Let $i_{1} < i_{2} < \cdots < i_{k}
\leq 2N$ denote those (possibly non-existent) indices where
$\mu_{i_j} < 0$ has an eigenvector $\x_{i_j} \in E^N$.
The two lemmas then combine to show:
\begin{corollary}
\label{cor:char_poly} Let $\beta > 0$. Then
\begin{equation*}
\det(L - \lambda \Id) = \lambda^{a_{\mathbf{M},\perp}(0) + a_\mathbf{M}(0)}
\Pi^{k}_{j=1} (\lambda^2 - \mu_{i_j})p_{\beta}(\lambda).
\label{eq:char_poly}
\end{equation*}
The roots of the polynomial $p_{\beta}(\lambda)$ all have non-zero
real part.
\end{corollary}
This corollary, along with the formula \eqref{eq:lambda}, suffice
to establish the desired result:
\begin{theorem}\label{thm:instability1}
The linearized second order system around the flock ring solution
\eqref{model} has an eigenvalue with positive real part if and
only if the linearized first order system around the ring solution
has a positive eigenvalue. As a consequence, the flock ring
solution is unstable for $m$-mode perturbations for the second
order model \eqref{model} if and only if the ring solution is
unstable for $m$-mode perturbations for the first order model
\eqref{eq:firstorder}.
\end{theorem}
\begin{proof}
Suppose first that $\mu_1 \leq 0$. Then $\x^*\mathbf{M}\x \leq 0$ for any
$\x$, whence all eigenvalues $\lambda$ of $L$ have non-positive
real part due to \eqref{eq:lambda}. Conversely, suppose $\mu_1 >
0$ and let $\mathcal{A}$ denote the set
\begin{equation*}
\mathcal{A} := \left\{ \beta \in [0,\infty) : \underset{  \lambda
\in \sigma(L)  }{\max} \Re(\lambda) > 0 \right\}.
\end{equation*}
Note that $0 \in \mathcal{A}$ due to \eqref{eq:lambda}. Indeed,
then $(\x_1,\sqrt{\mu_1}\x_1)$ defines an eigenvector with
eigenvalue $\lambda = \sqrt{\mu_1} > 0$. By continuous dependence
of the eigenvalues of $L$ on $\beta$, it follows that
$\mathcal{A}$ is relatively open. To show that it is also
relatively closed, let $\beta_{l} \in \mathcal{A}$ and $\beta_{l}
\rightarrow \beta_{0} \in (0,\infty)$. Up to extraction of
subsequences, it follows that there exists a corresponding
sequence $\lambda_{l}$ of eigenvalues with $\Re(\lambda_l) > 0$
converging to some $\lambda_{0}$ with $\Re(\lambda_0) \geq 0$.
Moreover, by continuous dependence of the coefficients of
$p_{\beta}(\lambda)$ on $\beta$, the roots of
$p_{\beta_l}(\lambda)$ converge to roots of
$p_{\beta_0}(\lambda)$. Thus $p_{\beta_0}(\lambda_0) = 0$. As no
such root can have zero real part by corollary
\ref{cor:char_poly}, $\Re(\lambda_0) > 0$ and $\beta_0 \in
\mathcal{A}$. As $\mathcal{A} \neq \emptyset$ it follows that
$\mathcal{A} = [0,\infty)$ as desired. The last part of the
theorem is a direct application of the first part to the $4x4$
$m$-mode perturbation matrix in \eqref{eq:bigmatrix1}.
\end{proof}
\begin{remark}\label{cor:stability12}
As an artifact of translation invariance in the first order model, the vector defined by
$\mathbf{e}_{2} := (0,1,\ldots,0,1)^{T} \in \mathbb{R}^{2N}$
always defines an eigenvector of $\mathbf{M}$ with eigenvalue
zero. Due to the fact that $\mathbf{e}_2 \in E^N,$ Lemma {\rm
\ref{lem:gen_eivecs}} implies that $(\mathbf{e}_2,\mathbf{e}_2)$
furnishes a generalized eigenvector with eigenvalue zero, so that
the flock is always linearly unstable for the model \eqref{model}.
\end{remark}

\subsection{Numerical validations}\label{sec:NumFlock}
In this section, we perform some numerical computations to show
stability regions for the flock ring. Moreover, we will show the
formation of clusters and the fattening instability. Due to
Theorem~\ref{thm:instability1} and
Corollary~\ref{cor:stability12}, applied to the $4\times 4$ matrix
in \eqref{eq:bigmatrix1}, we are reduced to study the determinant
of the matrix $M$ for the flock stability. Note that for fixed
values of $N$ and $m$ the determinant of $M$ is a function of the
parameters $a$ and $b$ and such that
\begin{align*}
D(a,b):=&\det(M)=I_1(m)I_1(-m)-(I_2(m))^2,\\
T(a,b):=&{\rm trace}(M)=I_1(m)+I_1(-m).
\end{align*}
\begin{remark}
Using the results of {\rm \cite[Theorem 3.1]{BUKB2}} one is able to
estimate the asymptotic value of the determinant of $M$.
In our case, using $W(x)=\frac{|x|^a}{a}-\frac{|x|^b}{b}$ one obtains that
\[
    \det(M)\sim C m^{-b+1} \quad\text{as}\quad m\to\infty,
\]
where $C>0$ and $b\in(1,2)\cup(4,6)\cup(8,10)\cup\cdots$. In these
cases $\det(M)>0$ and ${\rm trace}(M)<0$.
Moreover, this result shows that there is no spectral gap for
large modes $m$ since $\det(M)\to0$ as $m\to\infty$.
\end{remark}
In Figure~\ref{fig:detmodes} we recover some results on the
stability already shown in \cite{KSUB} . Since $I_1(m)$ and
$I_2(m)$ depend on the powers $a$ and $b$ of the power law
potential, we plot in the parameter region $\lbrace (a,b)\,:\,
a>b>0\rbrace$ the stability and instability regions depending on
the determinant of $M$ and its trace. We show the cases $m=3, 4,
5$ for a fixed $N=100000$.
\begin{figure}[ht]
\centering
\includegraphics[scale=0.29]{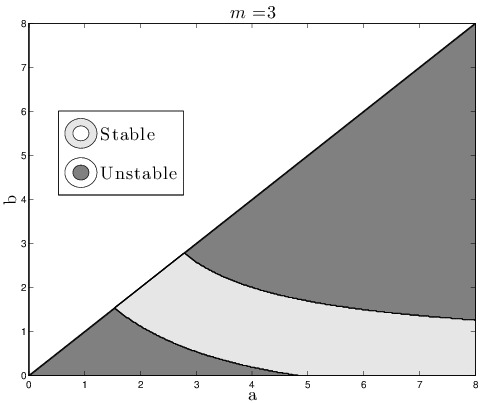}
\includegraphics[scale=0.29]{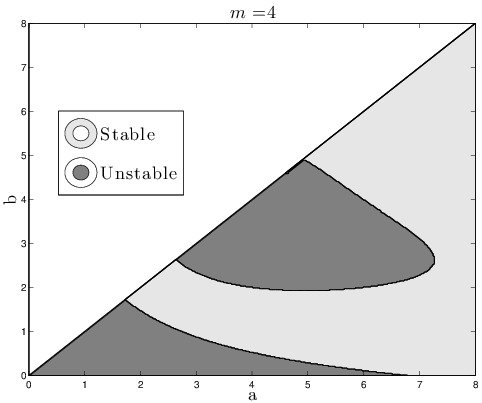}
\includegraphics[scale=0.29]{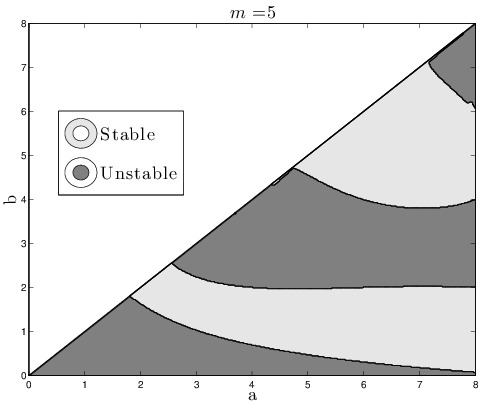}
\hfill \caption{Stability regions for different perturbations
modes $m=3, 4, 5$. Stability regions correspond to a couple $(a,b)$ of
parameters where $D(a,b)>0,\quad T(a,b)<0$.}
\label{fig:detmodes}
\end{figure}
In Figure~\ref{fig:stabreg} we compute the stability
area as a function of $a$ and $b$. To do so, we compute the
intersection of all stability areas for $m\geq 2$. It can be
observed from our tests that the stability area shrinks when the
number of particles increases. Moreover, it is observed that in
the limit when $N\to\infty$, the lower boundary of the stability
region converges to the dashed line. The red dashed line is the
curve $b=\frac{a}{a-1}$ that corresponds to the $m=+\infty$ mode.
This curve is the separatrix of the ins/stability regions for the
continuous delta ring of the first order continuum model, studied
in \cite{KSUB,BCLR}.
\begin{figure}[h!t]
\centering
\includegraphics[scale=0.29]{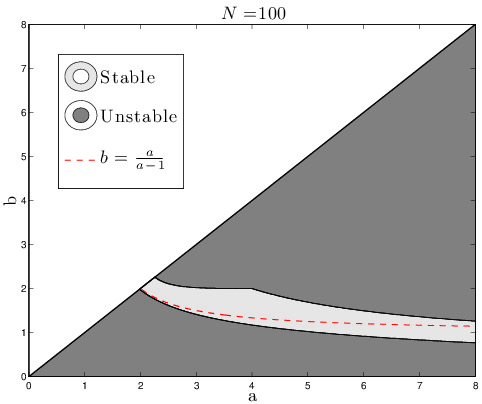}
\includegraphics[scale=0.29]{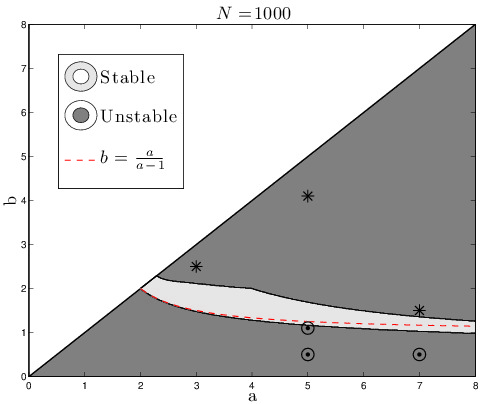}
\includegraphics[scale=0.29]{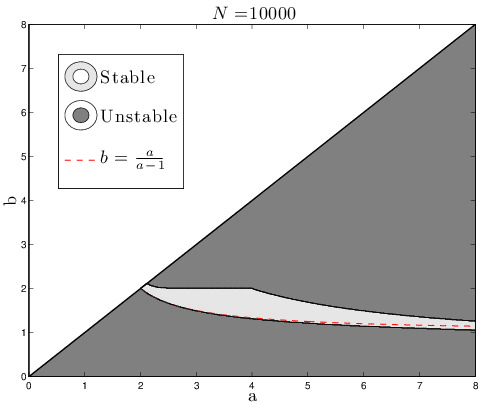}\hfill
\caption{Stability areas for flock ring solutions for different
values of $N$. From left to right: $N=100, 1000,
10000$. Markers $(\bigast)$ and $(\odot)$ indicate the explored parameters respectively in Table~\ref{tab:FlockClustering} and Table~\ref{tab:exfatflocks}.}
\label{fig:stabreg}
\end{figure}

\subsubsection{Cluster formation}
The formation of clusters occurs when the repulsion strength is
small. In other words, this phenomenon depends on how
singular the potential is at the origin. We show the bifurcation
diagram for the phase transition between equally distributed
flock and flock with cluster formation.
\begin{figure}[h!t!]
\centering
\includegraphics[scale=0.5]{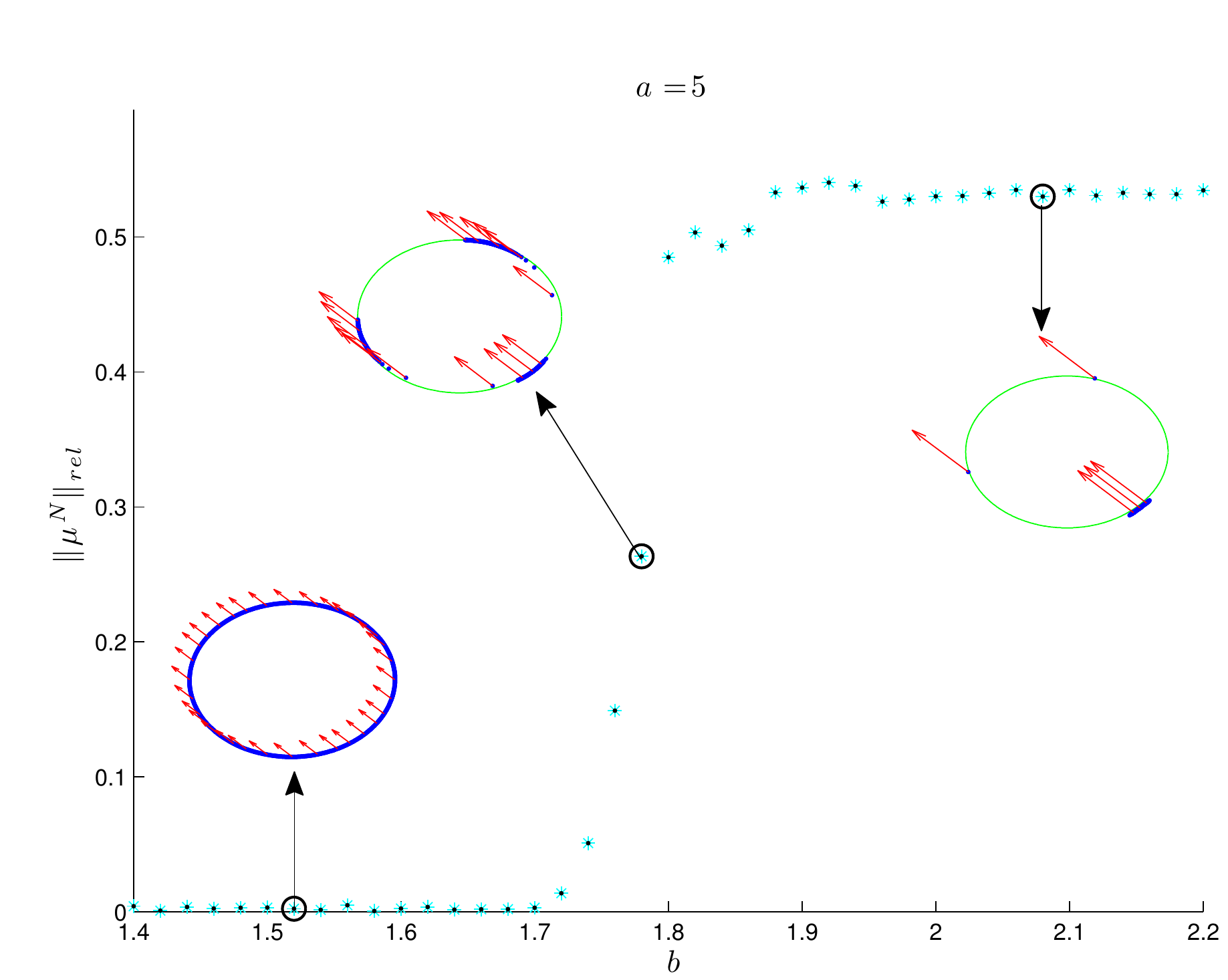}
\caption{Bifurcation diagram for cluster formation at $T_f=500$,
with  $N=1000$ particles, $a=5$, $|u_0|=2.5$.}
\label{fig:FlockClusteringbif}
\end{figure}
Figure~\ref{fig:FlockClusteringbif} is performed using $N=1000$
particles equally distributed on the stable circle with all the
velocities aligned. We let them evolve until $T_f=500$. We
fix the parameters $\abs{u_0}=2.5$, $a=5$ and vary $b$ along the
axis. The vertical axis represents the increment of the relative errors
\[
   \| \mu^N\|_{rel}=\frac{\| \mu^N-\mu^N_0 \|_2}{\| \mu^N_0 \|_2}
\]
with increasing $b$, where $\mu_0^N$ is the uniform distribution
along the stable ring of $N$ particles and $\mu^N$ the
distribution at time $T_f$. Simulations are performed with
\textit{MATLAB} and the evolution of the system of odes is solved
with the \textit{ode45} routine with adaptive time step.
Table~\ref{tab:FlockClustering} illustrates different possible
final states for different choices of the parameters $a, b$.
\begin{table}[h!]
\centering
\begin{tabular}{>{\centering}m{3.5cm}>{\centering}m{3.5cm}>{\centering}m{3.5cm}}
 $a=3, b=2.5$ & $a=5, b=4.1$  & $a=7, b=1.5$
\tabularnewline
\tabularnewline
\includegraphics[scale=0.25]{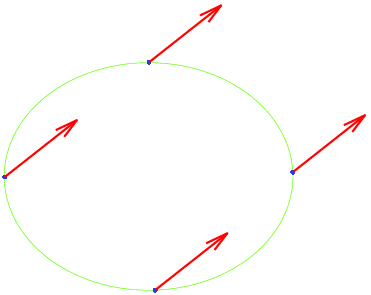}
&
\includegraphics[scale=0.25]{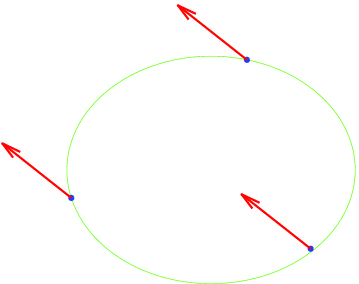}
&
\includegraphics[scale=0.25]{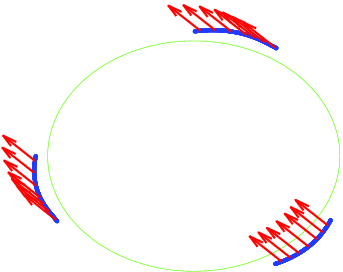}
\tabularnewline
\end{tabular}
\caption{Long time simulations with $N=1000$ particles. The
location of parameter values are marked as $(\bigast)$ points in the central plot of Figure {\rm\ref{fig:stabreg}}.}
\label{tab:FlockClustering}
\end{table}

\subsubsection{Fattening formation}
We show the transition diagram between a flock on a ring and a flock
on an annulus. In this case, the fattening phenomenon occurs when
the parameters of the potential cross the lower boundary of the
stability region. We numerically characterize this behavior in a
similar way as in the previous subsection.

\begin{figure}[h!]
\centering
\includegraphics[scale=0.5]{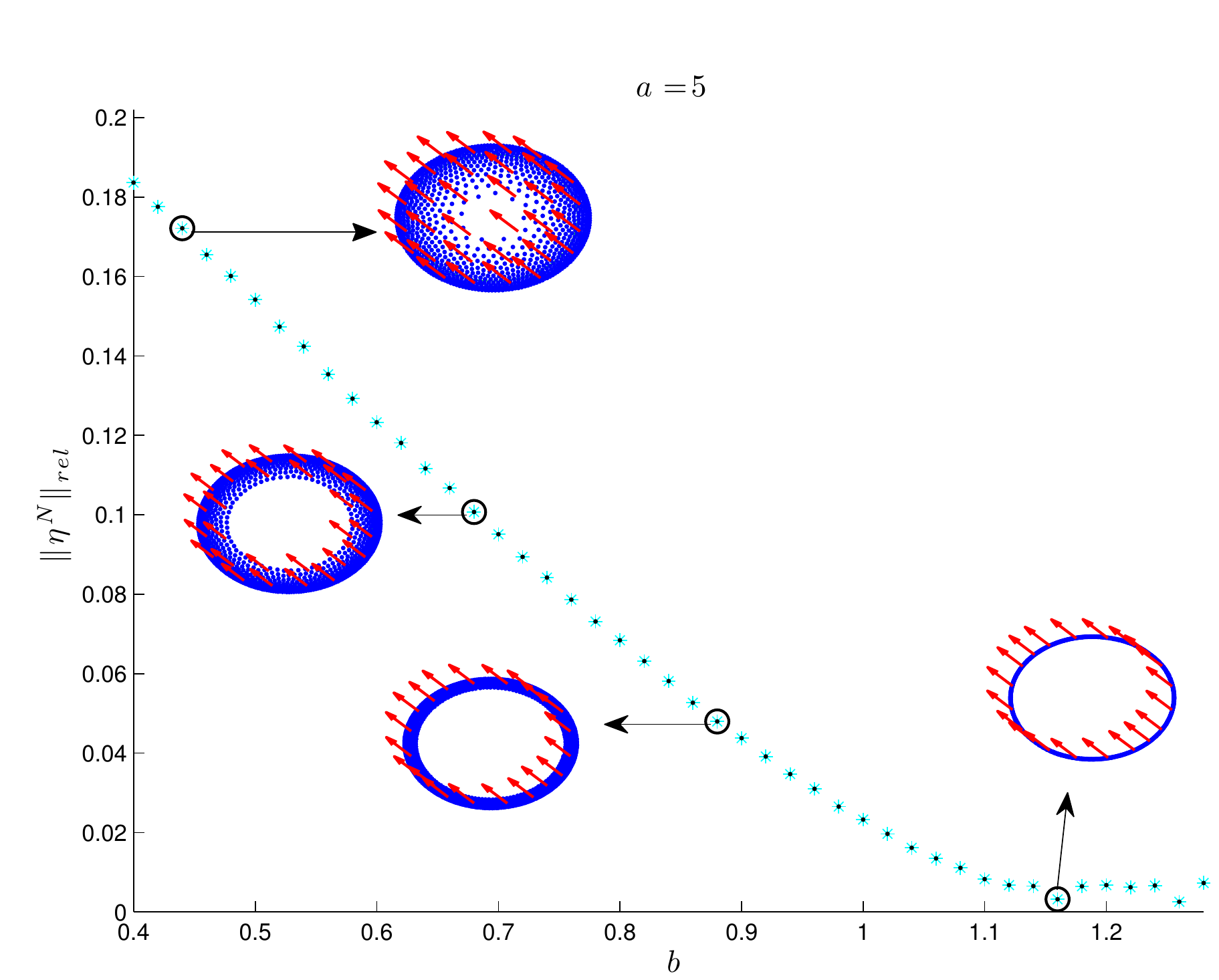}
\caption{Bifurcation diagram for fattening instability at
$T_f=500$ with  N=1000 particles, $a=5$, $|u_0|=2.5$.}
\label{fig:FlockFat}
\end{figure}

Figure~\ref{fig:FlockFat} is performed using $N=1000$ particles
equally distributed on the stable ring already in the steady state
with all the velocities aligned. We then let them evolve until
$T_f=500$. We fix the parameters $\abs{u_0}=2.5$, $a=5$ and vary $b$ along the axis. The vertical axis represents the increment of
the relative errors
\[
   \| \eta^N\|_{rel}=\frac{\| \eta^N-\eta^N_0 \|_2}{\| \eta^N_0 \|_2}
\]
for increasing $b$, where $\eta_0^N$ represents the average
distance from the center of mass for $N$ particles in a flock ring
formation, i.e., $\eta_0^N=R$ and $\eta^N$ is the average distance
from the center of mass at time $T_f$. Table~\ref{tab:exfatflocks}
shows the final states for some particular choices of parameters after
stabilization.
\begin{table}[h!]
\centering
\begin{tabular}{>{\centering}m{3.5cm}>{\centering}m{3.5cm}>{\centering}m{3.5cm}}
$a=5, b=1.1$ & $a=5, b=0.5$ & $a=7, b=0.5$
\tabularnewline
\tabularnewline
\includegraphics[scale=0.25]{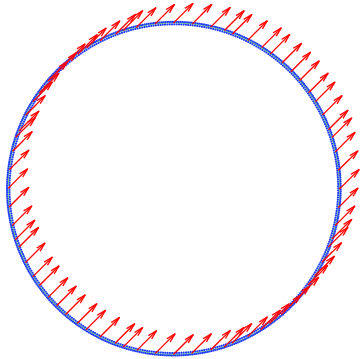}
&
\includegraphics[scale=0.25]{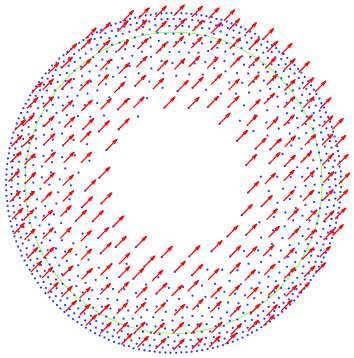}
&
\includegraphics[scale=0.25]{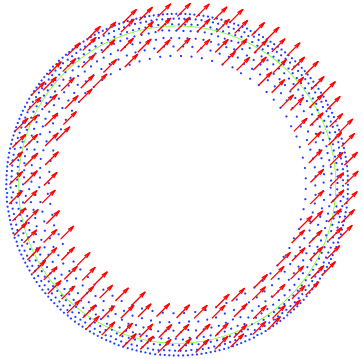}
\end{tabular}
\caption{Long time simulations with $N=1000$ particles. The
location of parameter values are marked as $(\odot)$ points in the
central plot of Figure {\rm\ref{fig:stabreg}}.}
\label{tab:exfatflocks}
\end{table}
\subsection{Stability of flock solutions with the Cucker-Smale alignment term}
As in Section~\ref{sec:nocuckersmale} we perform a study on
the flock stability for model \eqref{model2}. If we use the same change
of variables as in \eqref{eq:changeflock}, then system
(\ref{model2}) reads
\begin{equation}
\begin{cases}
\displaystyle\dot{y}_j=v_j-u_0=z_j
\\
\displaystyle\dot{z}_j=\frac{1}{N}\sum_{l=1}^N
H(\abs{y_l-y_j})(z_l-z_j)+\frac{1}{N}\sum_{\substack{l=1\\l\neq j}}^N
\nabla W(y_l-y_j)
\end{cases},\qquad\qquad\quad j=1,\dots, N
\label{eq:Aflock_noCS}
\end{equation}
We give a characterization of the flock solution in the complex
plane for \eqref{eq:Aflock_noCS} with $(y_{j}^{0},z_{j}^{0})=(Re^{i\theta_j},0),$
where $\theta_j=\frac{2\pi j }{N}$.
We consider then the perturbed solution
\begin{equation*}
\tilde{y}_j(t)=Re^{i\theta_j}(1+h_j(t)),\qquad
\end{equation*}
with $h_j$ such that $|h_j|\ll1$ and satisfying \eqref{eq:perturbationhj}.
Considering the following relations
\begin{align*}
&\tilde{y}_l-\tilde{y}_j=Re^{i\theta_j}\left(e^{\phi_p}h_l-h_j\right),
\\
&|\tilde{y}_l-\tilde{y}_j|\simeq2R\left|\sin\left(\frac{\phi_p}{2}\right)\right|+\frac{R}{4\left|\sin\left(\frac{\phi_p}{2}\right)\right|}\left[\left(1-e^{i\phi_p}\right)\left(h_l-\overline{h_j}\right)+\left(1-e^{-i\phi_p}\right)\left(\overline{h_l}-{h_j}\right)\right],
\\
&\tilde{z}_l-\tilde{z}_j=Re^{i\theta_j}\left(e^{\phi_p}h'_l-h'_j\right),
\end{align*}
where  $\phi_p=2\pi(l-j)/N=2\pi p/N$. We linearize the
Cucker-Smale alignment term around the solution up to first
order, leading to
\begin{align*}
H(|\tilde{y}_l-\tilde{y}_j|)\simeq&H(2R\left|\sin\left(\phi_p/2\right)\right|)\\
&+H'(2R\left|\sin\left(\phi_p/2\right)\right|)\frac{R}{4\left|\sin\left(\frac{\phi_p}{2}\right)\right|}\left[\left(1-e^{i\phi_p}\right)\left(h_l-\overline{h_j}\right)+\left(1-e^{-i\phi_p}\right)\left(\overline{h_l}-{h_j}\right)\right].
\end{align*}
Substituting the linearization in \eqref{eq:Aflock_noCS} and
neglecting the second order terms, we obtain the following
characterization of $h_j''$
\begin{align}
h_j''=&\frac{1}{N}\sum_{l=1
}^NH(2R|\sin\phi_p|)\left[e^{i\phi_p}h'_l-h'_j\right]\nonumber\\
&+\frac{1}{N}\sum_{\substack{l=1\\l\neq j}}^N\left[G_1(\phi_p/2)(h_j-e^{i\phi_p}h_l)+G_2(\phi_p/2)(\overline{h_l}-e^{i\phi_p}\overline{h_j})\right].
\label{eq:Perturb}
\end{align}
In order to study the behavior of the perturbations $h_j$, we
reduce the complexity of the problem, assuming that $h_j$
satisfies the following relation
\begin{equation*}
h_j=\xi_+(t)e^{im\theta_j}+\xi_{-}(t)e^{-im\theta_j}, \quad
h'_j=\xi'_+(t)e^{im\theta_j}+\xi'_{-}(t)e^{-im\theta_j},\qquad
m\in \mathbb{N}.
\end{equation*}
Therefore, we can express $h_l$ in terms of $h_j$ as
\begin{equation*}
h_l=\xi_+(t)e^{im\theta_j}e^{im\phi_p}+\xi_{-}(t)e^{-im\theta_j}e^{-im\phi_p},\qquad
m\in \mathbb{N}.
\end{equation*}
Inserting the previous expressions in \eqref{eq:Perturb} and
gathering terms in $e^{i\theta_j m}$ and $e^{-i\theta_j m}$, we
can characterize $\xi_+$ and $\xi_-$ as
\begin{align*}
\xi_+''=&\frac{1}{N}\sum_{l=1
}^NH(2R|\sin\phi_p|)\left[e^{i\phi_p(m+1)}-1\right]\xi'_+
+I_1(m)\xi_+ +I_2(m)\overline{\xi}_-, \\
\overline{\xi}_-''=&\frac{1}{N}\sum_{l=1
}^NH(2R|\sin\phi_p|)\left[e^{i\phi_p(m-1)}-1\right]\overline{\xi}_-
+I_2(m)\xi_+ +I_1(-m)\overline{\xi}_-,
\end{align*}
where $I_1$ and $I_2$ are defined in \eqref{eq:I1} and
\eqref{eq:I2}. Through a simple manipulation of the sum for the
linearized \emph{Cucker-Smale} term, we obtain that the expression
\begin{multline*}
\frac{1}{N}\sum_{l\neq j }H(2R|\sin\phi_p|)\left[e^{i\phi_p(m\pm1)}-1\right]=\\
\frac{1}{N}\sum_{l\neq j
}H(2R|\sin\phi_p|)\left[\cos(\phi_p(m\pm1))-1\right]+\frac{i}{N}\sum_{l\neq
j }H(2R|\sin\phi_p|)\sin(\phi_p(m\pm1)),
\end{multline*}
is real. Actually, $H(2R|\sin\phi_p|)$ and $\sin(\phi_p(m\pm1))$
are respectively symmetric and antisymmetric with respect to the
values of $\phi_p$, so the imaginary part vanishes. Recalling
the definition of $\phi_p$, we conclude
\begin{align*}
J_{\pm}(m)=&\frac{1}{N}\sum_{k=1}^{N}H\left(2R\left|\sin\left(\frac{2\pi p}{N}\right)\right|\right)\left[\cos\left(\frac{2\pi p}{N}(m\pm1)\right)-1\right]\nonumber\\
=&-\frac{4}{N}\sum_{k=1}^{N/2}H\left(2R\left|\sin\left(\frac{2\pi
p}{N}\right)\right|\right)\left[\sin^2\left(\frac{\pi
p}{N}(m\pm1)\right)\right].
\end{align*}
Therefore, we reduce the stability analysis to the following system
\[
\begin{pmatrix}
\xi_+''
\\\\
\overline{\xi}_-''
\end{pmatrix}
= \underbrace{\begin{pmatrix} I_1(m) & I_2(m)
\\\\
I_2(m) & I_1(-m)
\end{pmatrix}}_{M}
\begin{pmatrix}
\xi_+
\\\\
\overline{\xi}_-
\end{pmatrix}
+ \underbrace{\begin{pmatrix} J_+(m) & 0
\\\\
0 & J_-(m)
\end{pmatrix}}_{J}
\begin{pmatrix}
\xi'_+
\\\\
\bar{\xi'_-}
\end{pmatrix}.
\]
Taking the conjugate in the second equation and relabeling
$\overline{\xi}_-$ with ${\xi}_-$ as in \cite{BUKB2}, the previous
system is equivalent to
\begin{equation}\label{eq:systemCS}
\frac{d}{dt}
\begin{pmatrix}
\xi_+
\\
{\xi}_-
\\
\eta_+
\\
{\eta}_-
\end{pmatrix}
=
\begin{pmatrix}
0 &0& 1& 0 \\
0 &0& 0& 1\\
I_1(m) & I_2(m)  & J_+(m) & 0\\
I_2(m) & I_1(-m) & 0         & J_-(m)\\
\end{pmatrix}
\begin{pmatrix}
\xi_+
\\
{\xi}_-
\\
\eta_+
\\
{\eta}_-
\end{pmatrix}=
\begin{pmatrix}
0 & \Id \\
M & J
\end{pmatrix}
\begin{pmatrix}
\xi_+
\\
{\xi}_-
\\
\eta_+
\\
{\eta}_-
\end{pmatrix},
\end{equation}
where $\eta_\pm=\xi'_\pm$.

At this point, we will do a stability analysis based on the
eigenvalues of the matrix of the previous system in a similar way
as in Section~\ref{sec:nocuckersmale}. If instead of the
self-propelled/ friction term we use a Cucker-Smale type alignment
term
\begin{equation*}
-\sum^{N}_{l=1} g\left(|x_j - x_l| \right)(v_j - v_l),
\end{equation*}
where $g(r)$ denotes any strictly positive function, the corresponding stability matrix $L_{{\rm CS}}$ for the flock reads
\begin{equation*}
L_{{ \rm CS } } = \begin{pmatrix} 0 & \mathrm{Id} \\ \mathbf{M} & -G \end{pmatrix}.
\end{equation*}
As before, $\mathbf{M}$ denotes stability matrix of the
first order model. As the alignment term is linear in the
velocity, the matrix $G$ acts on $\vv = (v_1,\ldots,v_{N})^{T}, \;
v_j \in \mathbb{R}^2,$ according to the relation
\begin{equation*}
(G\vv)_{j} = \sum^{N}_{l=1} g\left(|x_j - x_l| \right)(v_j - v_l).
\end{equation*}
In particular, if we denote $||v_j - v_l||^{2}_{2} := (v_{j} -
v_{l})^{*}(v_j-v_l)$ then this relation implies that
\begin{equation*}
\vv^{*} G \vv  = \frac{1}{2}\sum^{N}_{j,l = 1} g\left(|x_j - x_l|\right)|| v_j - v_l||^{2}_{2}.
\end{equation*}
Consequently, $G$ is positive semi-definite and $G\vv = \0$ if and
only if $\vv$ is ``constant'' in the sense that $v_{j} \equiv w$
for some fixed $w \in \mathbb{R}^{2}$. In other words,
$\mathrm{ker}(G) = \mathrm{span} \left\{\mathbf{e}_1, \mathbf{e}_2
\right\}$. By translation invariance of the first order model,
both $\mathbf{e}_1 \in \ker(\mathbf{M})$ and $\mathbf{e}_2 \in
\ker(\mathbf{M})$ as well.

Note that the eigenvalue problem for $L_{ {\rm CS} }$ is again
equivalent to the following quadratic eigenvalue problem for $\x
\in \mathbb{C}^{2N}$: $\lambda^{2} \x + \lambda G \x - \mathbf{M}
\x = \0$. Assuming the normalization $\x^*\x = 1$, the quadratic
formula then implies that
\begin{equation*}
\lambda = \frac{ - \x^*G\x \pm \sqrt{ (\x^*G\x)^2 + 4 \x^*\mathbf{M} \x } }{2}.
\end{equation*}
From this relation and the fact that $\mathrm{ker}(G) \subset
\mathrm{ker}(\mathbf{M})$ we conclude that $\mathrm{ker}(L_{ {\rm
CS} }) = \mathrm{ker}(\mathbf{M})$, and moreover that
$\Re(\lambda) = 0 \quad \Leftrightarrow \quad \lambda = 0$.
Furthermore, $\mathbf{e}_1$ and $\mathbf{e}_2$ generate a single
generalized eigenvector whereas each remaining $\x \in
\ker(\mathbf{M})$ generates no generalized eigenvectors. Indeed,
corresponding to each $\x \in \ker(\mathbf{M})$ the system of
equations
\begin{equation*}
\begin{pmatrix} 0 & \mathrm{Id} \\ \mathbf{M} & -G \end{pmatrix} \begin{pmatrix} \uu \\ \ww \end{pmatrix} = \begin{pmatrix} \x \\ \0 \end{pmatrix}.
\end{equation*}
has a solution if and only if $G\x = \0$ and $\uu \in
\ker(\mathbf{M})$. Additionally, if $\x = \mathbf{e}_i$ then for
any $\uu \in \ker(\mathbf{M})$ the system of equations
\begin{equation*}
\begin{pmatrix} 0 & \mathrm{Id} \\ \mathbf{M} & -G \end{pmatrix} \begin{pmatrix} \tilde{\uu} \\ \tilde{\ww} \end{pmatrix} = \begin{pmatrix} \uu \\ \x \end{pmatrix}
\end{equation*}
has no solutions. This follows by multiplying the second equation
by $\x^*$, then using the facts that $\mathbf{e}_i \in \ker(G)
\subset \ker(\mathbf{M})$ and the facts that $G$ and $\mathbf{M}$
are symmetric. In other words, if $g(r)$ is any strictly positive
function then
\begin{equation*}
\det( L_{{\rm CS}} - \lambda \mathrm{Id} ) = \lambda^{2+\dim(\ker(\mathbf{M}))}p_{g}(\lambda),
\end{equation*}
for some polynomial $p_{g}(\lambda)$ that has non-zero roots.
Since this equation holds for any strictly positive function
$g(r),$ we may follow the proof of Theorem \ref{thm:instability1}
to conclude that the second order model has an eigenvalue with
positive real part if and only if the first order system has a
positive eigenvalue. Moreover, the vectors
$(\mathbf{e}_i,\mathbf{e}_i)$ for each $i=1,2$ furnish generalized
eigenvectors with eigenvalue zero, so the ring flock is always
linearly unstable. As a summary, we have shown:
\begin{theorem}
The linearized second order system \eqref{model2} around the flock
ring solution has an eigenvalue with positive real part if and
only if the linearized first order system around the ring solution
has a positive eigenvalue. As a consequence, the flock ring
solution is unstable for $m$-mode perturbations for the second
order model \eqref{model2} if and only if the ring solution is
unstable for $m$-mode perturbations for the first order model
\eqref{eq:firstorder}.
\end{theorem}
The linear stability analysis of the previous system leads to the
characterization of the stability areas in
Figure~\ref{fig:stableareasCS}. We show that the stability
parameter regions for different values of $N$, and $\gamma=1$
coincide with the ones in Figure \ref{fig:stabreg}.
\begin{figure}[h!]
\centering
\includegraphics[scale=0.29]{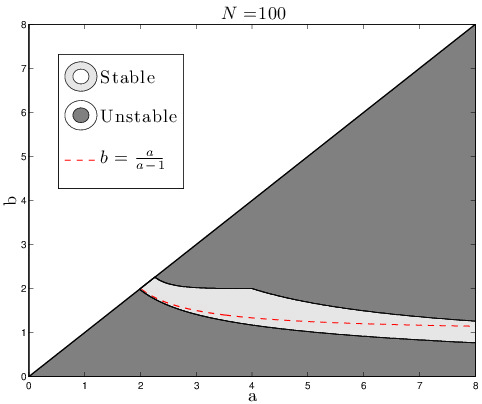}
\includegraphics[scale=0.29]{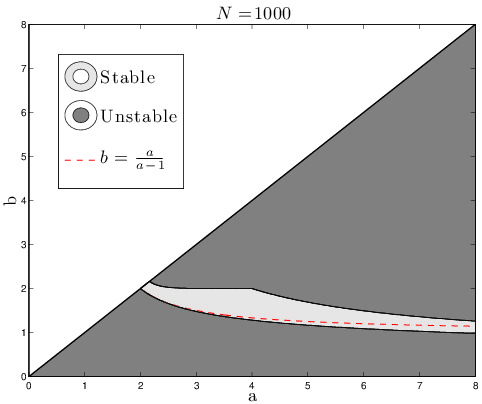}
\includegraphics[scale=0.29]{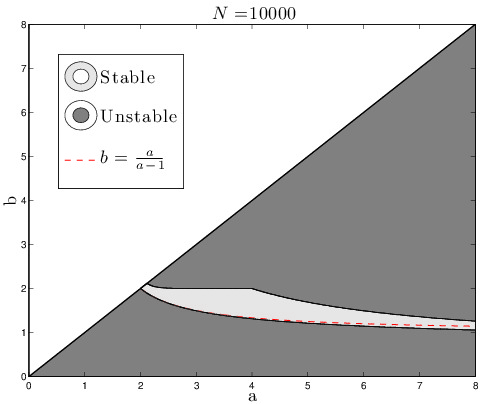}
\caption{Stability regions for flock ring solutions with the Cucker-Smale term for different values of $N$.}
\label{fig:stableareasCS}
\end{figure}

We also investigate the behavior of the eigenvalue with the
largest real part, $\Re(\lambda_1)$, of the linearized system
\eqref{eq:systemCS} against the increasing value of communication
strength $\gamma$. In Figure~\ref{fig:eigenvalueCS}, as the
potential gets more repulsive at the origin, we see the change
from stability to instability, and the rate of convergence to
equilibrium depending on $\gamma$.
\begin{figure}[h!]
\centering
\includegraphics[scale=0.4]{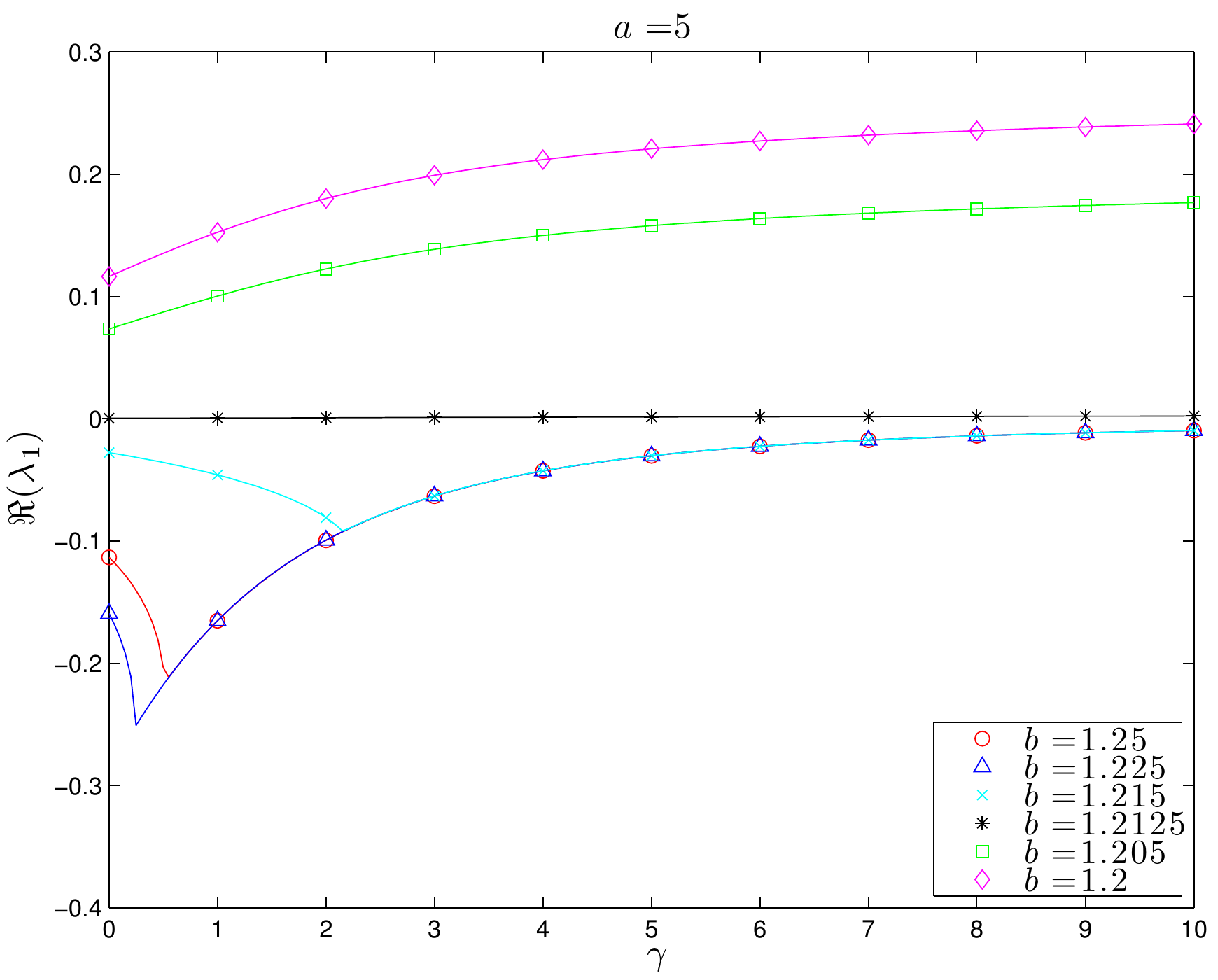}
\includegraphics[scale=0.4]{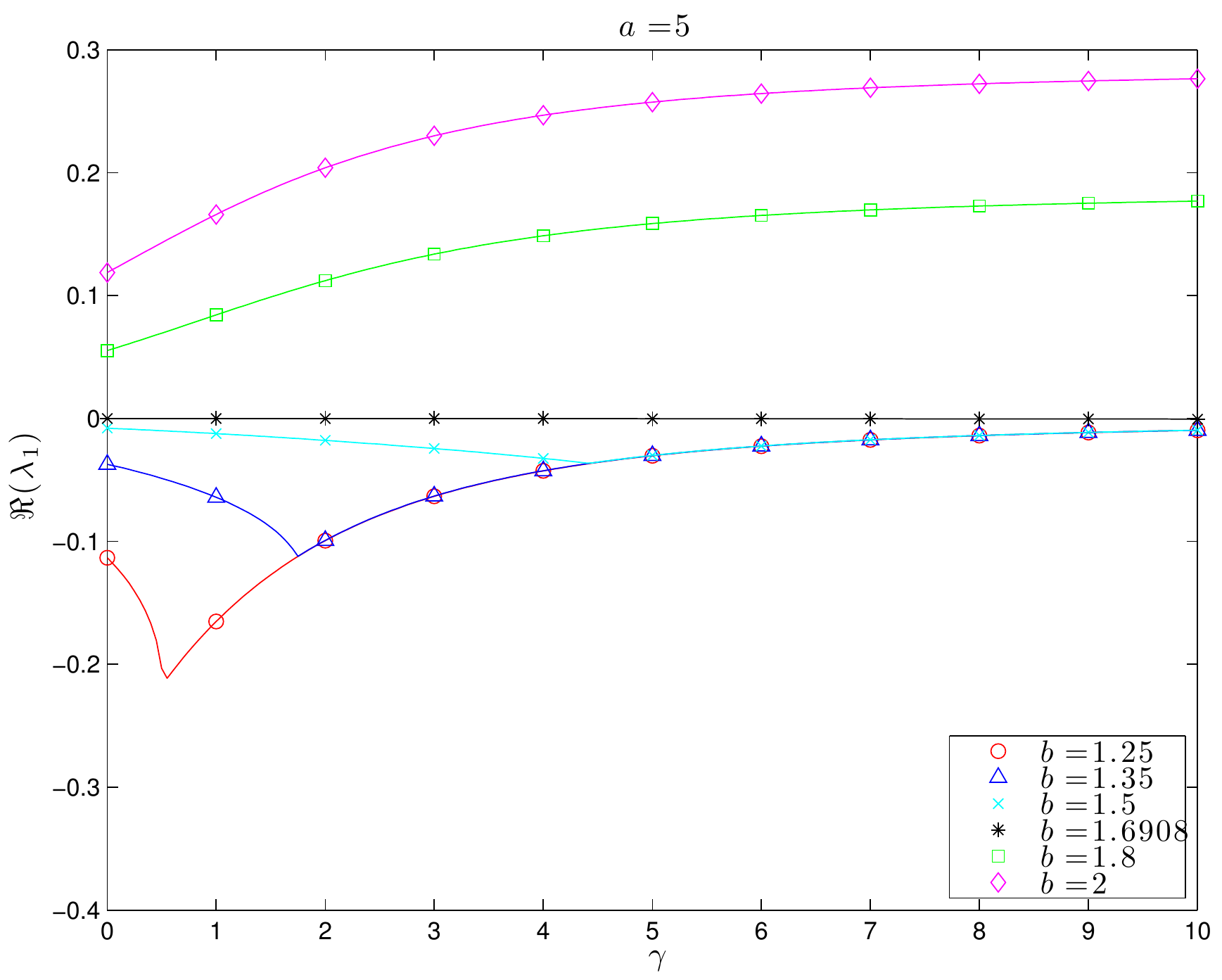}
\caption{The magnitude of $\Re(\lambda_1)$ is influenced by $\gamma$, for different values of $b$ and fixed $a=5$, $N=10000$.  }
\label{fig:eigenvalueCS}
\end{figure}
\section{Stability for mill solutions}\label{sec:millstability}
This section is devoted to complement the results in \cite{BUKB2} by analyzing the stability of mill ring solutions with repulsion.
\subsection{Linear stability analysis}

Let us consider the transformation
\[
\begin{cases}
    y_j(t)=O(t)x_j(t) \\
    z_j(t)=O(t)v_j(t)
\end{cases},\quad j=1,\dots, N
\]
where $O(t)$ is the rotation matrix defined as
\[
    O(t)=e^{St},\quad     S=\begin{pmatrix}0 & \omega\\ -\omega & 0 \end{pmatrix},\quad \mbox{and}\quad \dot{O}(t)=Se^{St}.
\]
Evaluating $\dot{y}_j(t)$ and $\dot{z}_j(t)$ and after some straightforward computations, we get
\[
\begin{cases}
\displaystyle\dot{y}_j(t)=Sy_j(t)+z_j(t)\\
\displaystyle\dot{z}_j(t)=Sz_j(t)+(\alpha-\beta|z_j|^2)z_j(t)-\frac{1}{N}\sum_{\substack{l=1\\l\neq j}}^N\nabla W(y_l-y_j)
\end{cases},\quad j=1,\dots, N.
\]
A linear stability analysis for mill rings was performed in
\cite{BUKB2}. Actually, for a fixed number of particles, we have a
mill ring solution given by
$(y_{j}^{0},z_{j}^{0})=(Re^{i\theta_j},0),$ where
$\theta_j=\frac{2\pi j }{N}$, and $R$ determined by equation
\eqref{eq:radiuseq}. With the same notation as in
Section~\ref{sec:linear_flock}, the analysis in \cite{BUKB2} leads
to the linear system
\begin{equation}
\begin{pmatrix}
\xi'_+\\{\xi}'_-\\\eta'_+\\{\eta}'_-
\end{pmatrix} =
\begin{pmatrix}
0 & 0 & 1 & 0
\\
0 & 0 & 0 & 1
\\
 -\omega i\alpha+\omega^2+I_1(m) &
-\omega i\alpha+I_2(m) & -\alpha-2\omega i & \alpha
\\
\omega i\alpha+I_2(m) &  \omega
i\alpha+\omega^2+I_1(-m) & \alpha& -\alpha+2\omega i
\end{pmatrix}
\begin{pmatrix}
\xi_+\\{\xi}_-\\\eta_+\\{\eta}_-
\end{pmatrix}.
\label{eq:systemMill}
\end{equation}
Let us remind that the perturbations are of the form
$\tilde{y}_j(t)=Re^{i\theta_j}(1+h_j(t))$, with
$h_j=\xi_+(t)e^{im\theta_j}+\xi_{-}(t)e^{-im\theta_j}$,
$m=2,3,\dots,$ such that $|h_j|\ll1$ and satisfying
\eqref{eq:perturbationhj}, with
$(\eta_+,{\eta}_-)=(\xi'_+,{\xi}'_-)$. We will make use of
\eqref{eq:systemMill} to study the stability of mill rings with
repulsion.

\subsection{Numerical tests}
Unlike the case of flock solutions where the asymptotic speed does
not play role in the linear stability, we will show that the
asymptotic speed $|u_0|$ can be used as a bifurcation parameter
for mills.

In Table~\ref{tab:SAreaMill} we numerically investigate the
behavior of the stability region for a fixed number of particles,
$N=1000$, and for increasing values of the asymptotic speed
$|u_0|$. We observe that the stability region shrinks with respect
to $a$ and gets larger with respect to $b$. Each stability region
in Table~\ref{tab:SAreaMill} is computed out of the intersection
of the stable areas for the system \eqref{eq:systemMill} for each
perturbation mode $m\ge 2$. Note that for $|u_0|=0$ the stability
region coincides with the one for the the first order model
\eqref{eq:firstorder} and for the flock ring solution.
\begin{table}[h!]
\begin{tabular}{>{\centering}m{4.55cm}>{\centering}m{4.55cm}>{\centering}m{4.55cm}}
\includegraphics[scale=0.295]{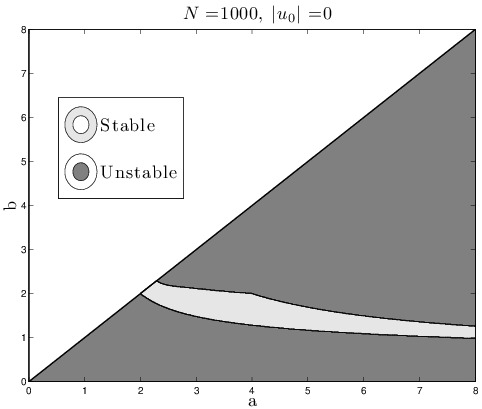}
&
\includegraphics[scale=0.295]{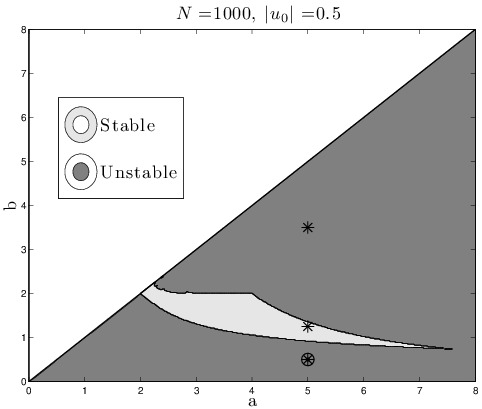}
&
\includegraphics[scale=0.295]{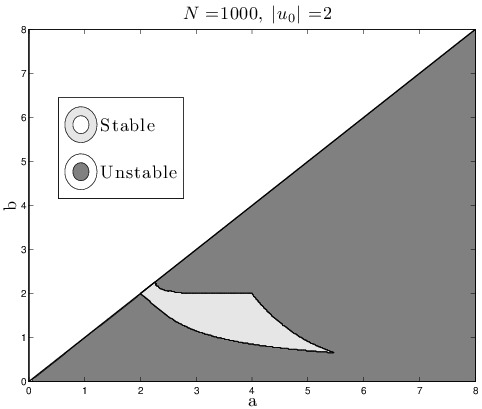}
\tabularnewline
\includegraphics[scale=0.295]{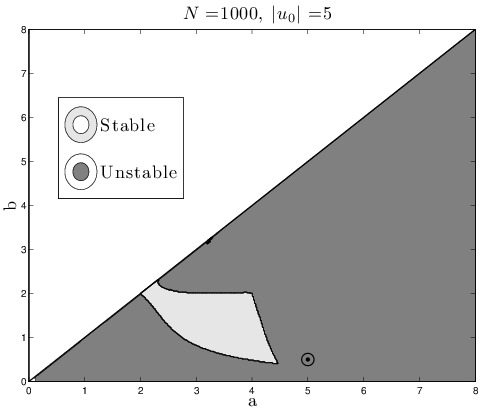}
&
\includegraphics[scale=0.295]{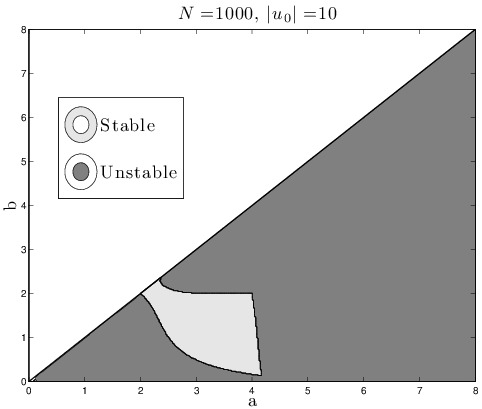}
&
\includegraphics[scale=0.295]{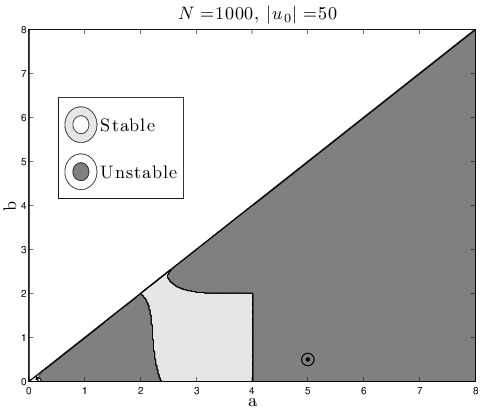}
\end{tabular}
\caption{Stability region for $N=1000$ and different values of the asymptotic speed $|u_0|$. Markers $(\odot)$ and $(\bigast)$ correspond to the explored parameters in Table~{\rm \ref{tab:fat2clust}} and Table~{\rm \ref{tab:fat2clustb}}.}
\label{tab:SAreaMill}
\end{table}

A similar analysis, as done in Subsection~\ref{sec:NumFlock}, can
be performed to study the formation of fat mills and clustered
mill solutions. We show how both the fattening and the clustering
instability are triggered by tunning the asymptotic speed for a
choice of the interaction potential ($a$ and $b$).

In the case of flock ring solutions we observe cluster solutions
or annulus solutions when parameters $a$ and $b$ are chosen
respectively ``below" or ``above" the stability region. In the
case of mill solutions, a similar behavior is observed, but this
will depend also on the chosen value of $|u_0|$. As an example, we
fix $(a,b)=(5,0.5)$, marked as $(\odot)$ in
Table~\ref{tab:SAreaMill}, and we observe the behavioral change of
the system for increasing values of the asymptotic speed.

Table~\ref{tab:fat2clust} exhibits this switching behavior from
a fat mill to a cluster pattern along with the increment of the
asymptotic speed. We observe that for small values of the
asymptotic speed fat mill solutions are stable patterns, but when
increasing the value of $\abs{u_0}$ the stable solutions form a
clustered mill. The first row shows the evolution of the system with
asymptotic speed $\abs{u_0}=0.25$ towards an annulus mill.
In the second row we take $\abs{u_0}=0.5$ the previous stable solution is reshaped to a fat clusters pattern.
The speed in the third row is switched to $\abs{u_0}=5$ and clusters on lines emerge as a stable configuration. Increasing the speed to $\abs{u_0}=50$ in the fourth row we can observe that clusters on ``points'' are stable solutions.

\begin{table}[h!]
\centering
\begin{tabular}{>{\centering}m{.7cm}>{\centering}m{3.25cm}>{\centering}m{3.25cm}>{\centering}m{3.25cm}>{\centering}m{3.25cm}}
$\abs{u_0}$ & &  &  & Stable
\tabularnewline
$0.25$
&
\includegraphics[scale=0.16]{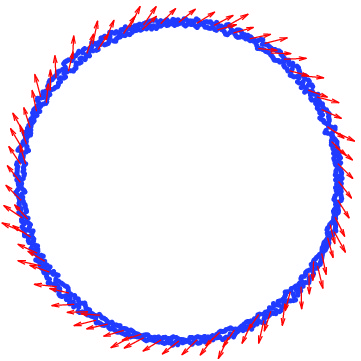}
&
\includegraphics[scale=0.16]{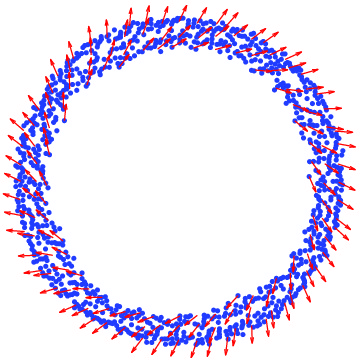}
&
\includegraphics[scale=0.16]{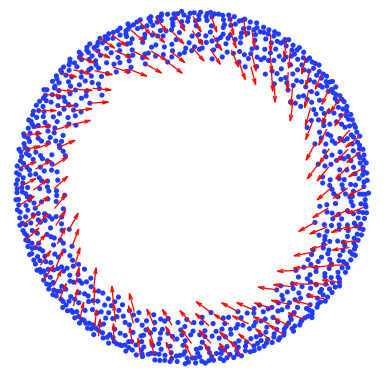}
&
\includegraphics[scale=0.16]{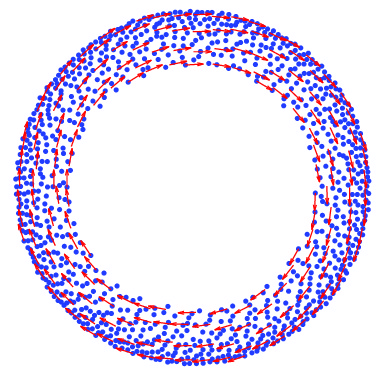}
\tabularnewline
$0.5$
&
\includegraphics[scale=0.16]{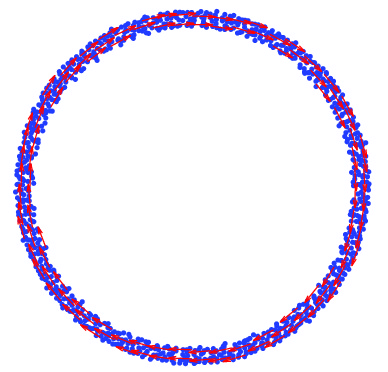}
&
\includegraphics[scale=0.16]{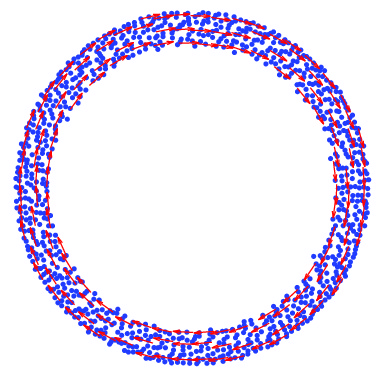}
&
\includegraphics[scale=0.16]{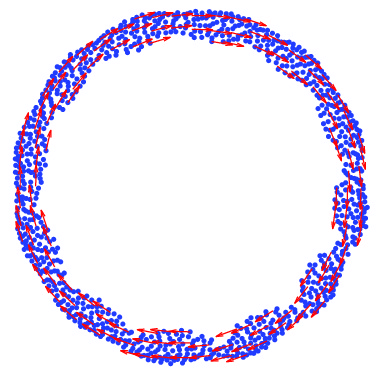}
&
\includegraphics[scale=0.16]{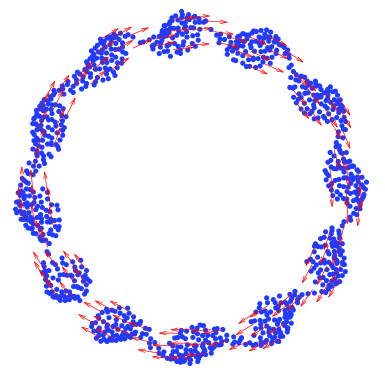}
\tabularnewline
$5$
&
\includegraphics[scale=0.16]{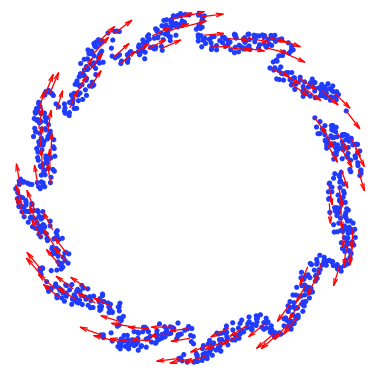}
&
\includegraphics[scale=0.16]{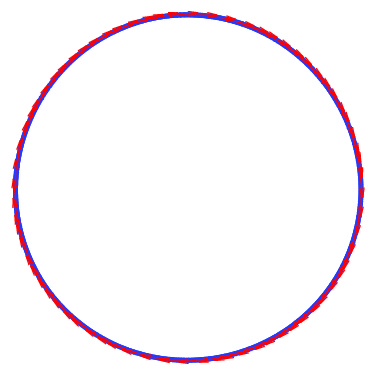}
&
\includegraphics[scale=0.16]{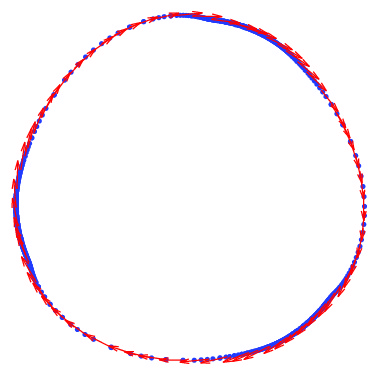}
&
\includegraphics[scale=0.16]{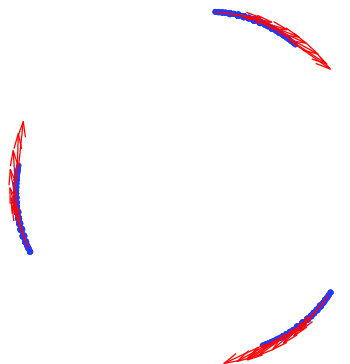}
\tabularnewline
$50$
&
\includegraphics[scale=0.16]{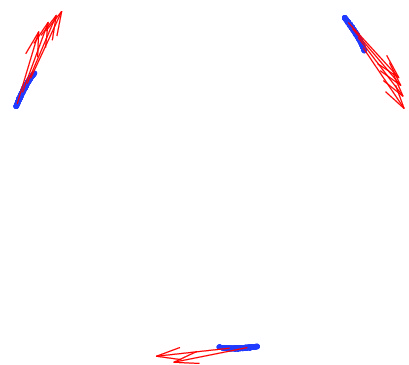}
&
\includegraphics[scale=0.16]{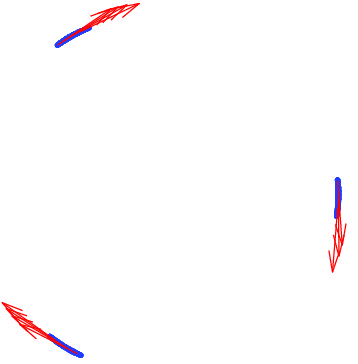}
&
\includegraphics[scale=0.16]{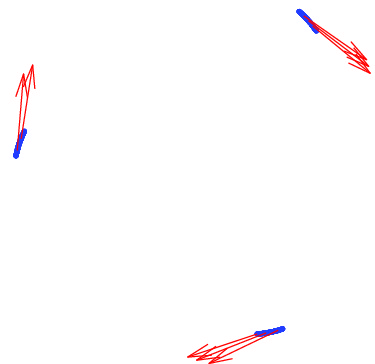}
&
\includegraphics[scale=0.16]{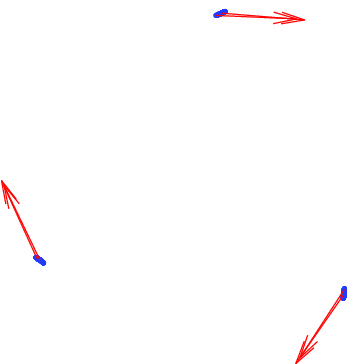}
\end{tabular}
\caption{$N=1000$ particles, $a=5, b=0.5$. The table shows the evolution of a mill ring for increasing values of the speed $\abs{u_0}$. Each row depicts the behavior of the system for a fixed speed, until a stable state is reached. The evolution of the second, third and fourth row is computed starting from the stable pattern of the previous line.}
\label{tab:fat2clust}
\end{table}
In Figure~\ref{fig:Stableu0b} we numerically show how the stability
region looks like in terms of $(|u_0|, b)$, with $a$ fixed at $0.5$, and we enlighten with marker points $(\odot)$ the parameter choices  of Table~\ref{tab:fat2clust}, first and second lines.

For the sake of completeness, we enrich the analysis fixing $\abs{u_0}=0.5$ and considering different values of $b$, in order to cross the stability region. Therefore in Table~\ref{tab:fat2clustb} we show the evolution of a mill ring solution with $b$ taken subsequently equal to $0.5, 1.25, 3.5$, parameter choices are marked as ($\bigast$) in Figure~\ref{fig:Stableu0b}. The first line of 
Table~\ref{tab:fat2clustb} shows the convergence to the same stable state as the one in second line of Table~\ref{tab:fat2clust}, but,  since the system starts to evolve directly from a ring mill solution, the transient behavior is  different.
Parameters in second line belong to the stability region, see Figure~\ref{fig:Stableu0b}. Therefore, the stable state becomes a mill ring solution.
Finally, in the third line we increase $b$ and a three point cluster solution is observed as stable pattern.
\begin{table}[h!]
\centering
\begin{tabular}{>{\centering}m{.7cm}>{\centering}m{3.25cm}>{\centering}m{3.25cm}>{\centering}m{3.25cm}>{\centering}m{3.25cm}}
$b$ & &  &  & Stable
\tabularnewline
$0.5$
&
\includegraphics[scale=0.16]{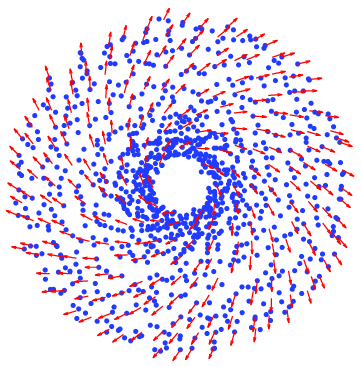}
&
\includegraphics[scale=0.16]{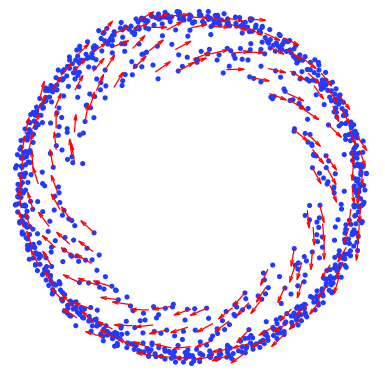}
&
\includegraphics[scale=0.16]{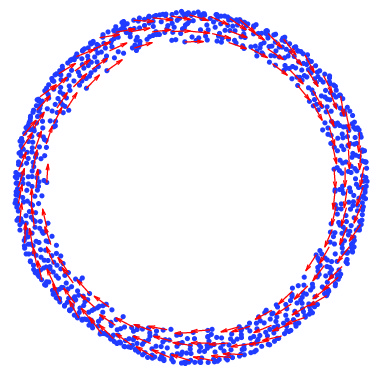}
&
\includegraphics[scale=0.16]{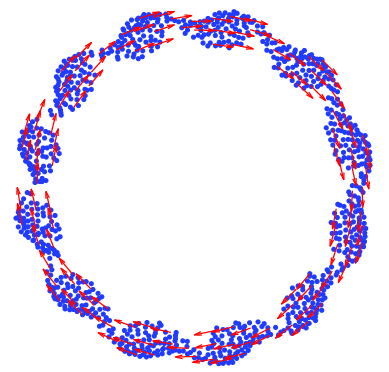}
\tabularnewline
$1.25$
&
\includegraphics[scale=0.16]{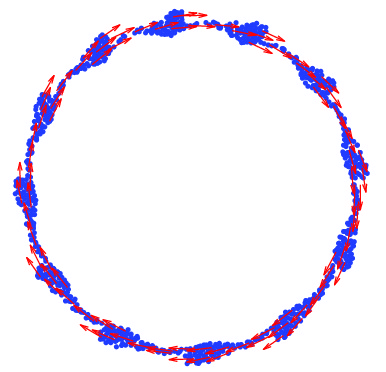}
&
\includegraphics[scale=0.16]{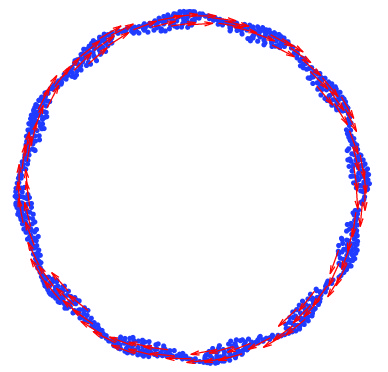}
&
\includegraphics[scale=0.16]{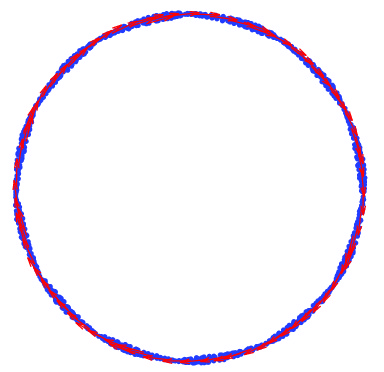}
&
\includegraphics[scale=0.16]{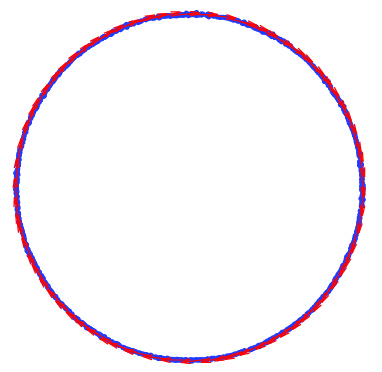}
\tabularnewline
$3.5$
&
\includegraphics[scale=0.16]{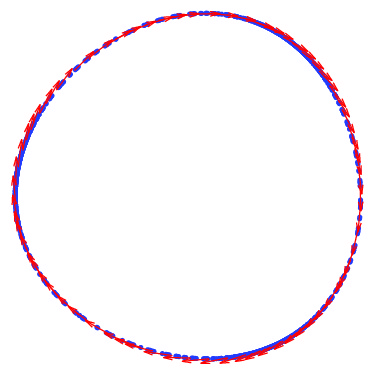}
&
\includegraphics[scale=0.16]{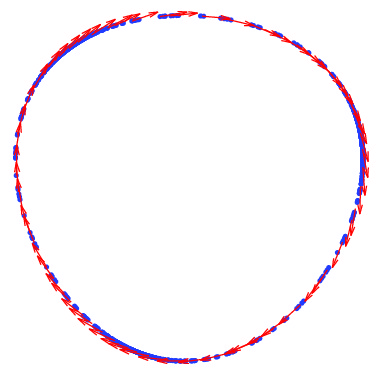}
&
\includegraphics[scale=0.16]{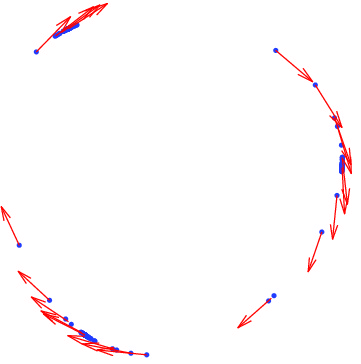}
&
\includegraphics[scale=0.16]{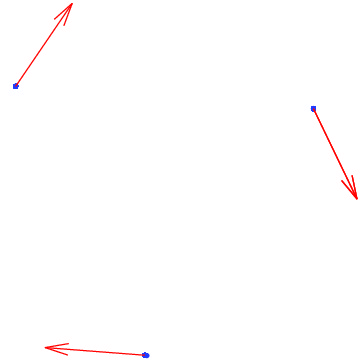}
\end{tabular}
\caption{$N=1000$ particles, $a=5, |u_0|=0.5$. The table shows the evolution of a mill ring for increasing values of  $b$, i.e. decreasing repulsion. The evolution of the second and the third row is computed starting from the stable pattern of the previous line.}
\label{tab:fat2clustb}
\end{table}
\begin{figure}[h!]
\centering
\includegraphics[scale=0.45]{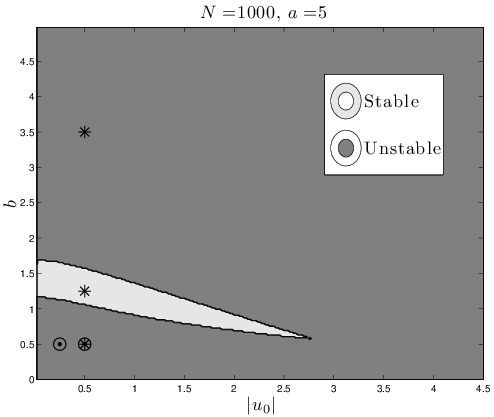}
\caption{Stability region for mill ring solution with parameter $a=5$, $N=1000$. Markers $(\odot)$ and $(\bigast)$ depict the parameter choices respectively for Table~{\rm \ref{tab:fat2clust}} and Table~{\rm \ref{tab:fat2clustb}}.}
\label{fig:Stableu0b}
\end{figure}
\subsubsection{Mill to Flock  and Flock to Mill behavior}

We numerically investigate the stability of mill and flock ring
solutions for small values of the asymptotic speed, $|u_0|$, and
the parameter $b$, which corresponds to a strong repulsion
condition.

We perform two representative simulations showing that for a
particular choice of the parameters,  mill ring solutions can
switch to fat flock solutions and conversely flock mill solutions
switch to fat mill patterns.

\begin{table}[h!]
\centering
\begin{tabular}{>{\centering}m{3.5cm}>{\centering}m{3.5cm}>{\centering}m{3.5cm}>{\centering}m{3.5cm}}
\includegraphics[scale=0.2]{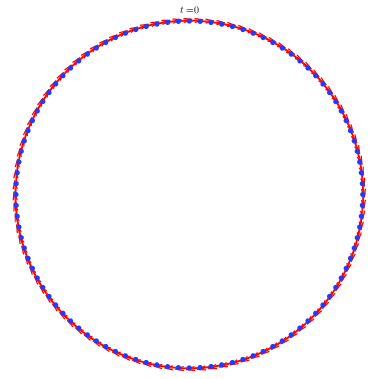}
&
\includegraphics[scale=0.2]{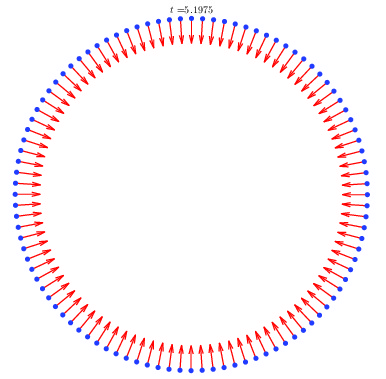}
&
\includegraphics[scale=0.2]{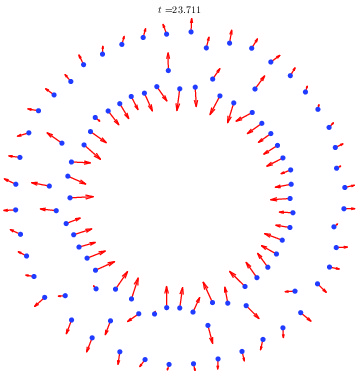}
&
\includegraphics[scale=0.2]{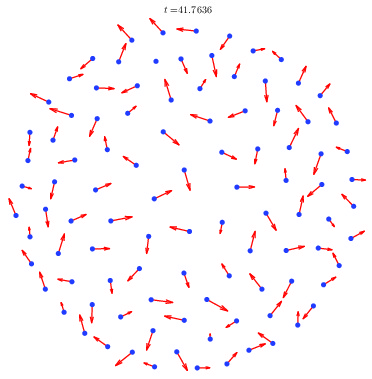}
\tabularnewline
\includegraphics[scale=0.2]{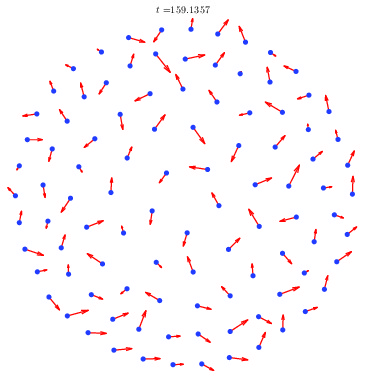}
&
\includegraphics[scale=0.2]{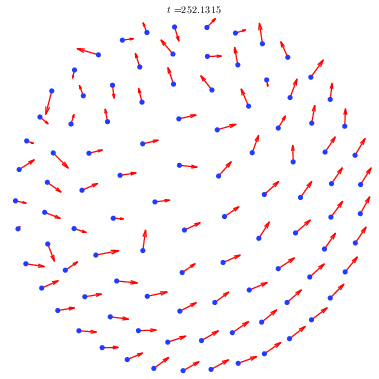}
&
\includegraphics[scale=0.2]{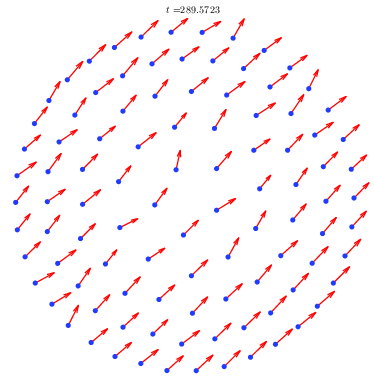}
&
\includegraphics[scale=0.2]{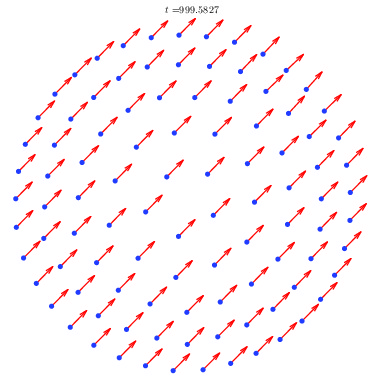}
\end{tabular}
\caption{System with $N=100$ agents, parameters are fixed $a=4$ and $b=0.0005$ and $|u_0|=0.01$. The first row shows that the initial mill ring configuration is unstable. The second row outlines the self organization of the system in a fat flock configuration.}
\label{tab:Mill2Flock}
\end{table}
In Table~\ref{tab:Mill2Flock} we take $N=100$ particles and we fix
$a=4, b=0.0005$ and $|u_0|=0.01$. The frames in the first row show
the instability of mill ring solutions for this choice of
parameters. The system initially evolves to an almost chaotic
state, then particles start to organize rotating around the center
of mass. This rotation actually causes the alignment of the agents
and the final fat flock configuration described in the second row.
In Table~\ref{tab:Flock2Mill} we consider as initial state a flock
ring solution. The parameters of the model are $N=100$, $a=4$,
$b=0.001$ and $|u_0|=0.1$. The first row of the table illustrates
that the initial configuration is not a stable solution.
Therefore, the symmetry of the flock ring is broken and the system
exhibits a chaotic behavior. In the second row a rotating dynamic
emerges out of the disordered state and finally the system
stabilizes to a fat mill solution.
\begin{table}[h!]
\centering
\begin{tabular}{>{\centering}m{3.5cm}>{\centering}m{3.5cm}>{\centering}m{3.5cm}>{\centering}m{3.5cm}}
\includegraphics[scale=0.2]{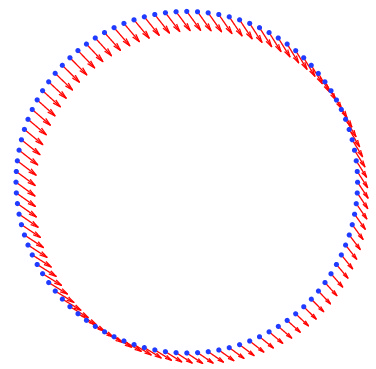}
&
\includegraphics[scale=0.2]{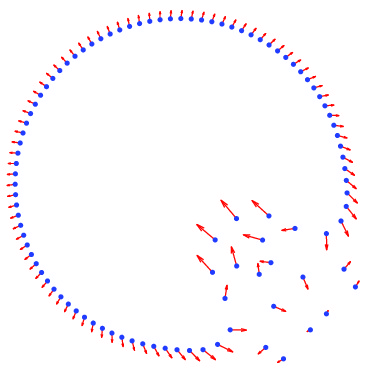}
&
\includegraphics[scale=0.2]{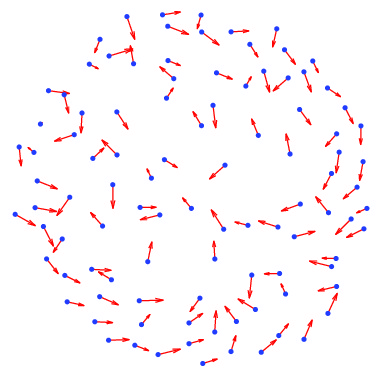}
&
\includegraphics[scale=0.2]{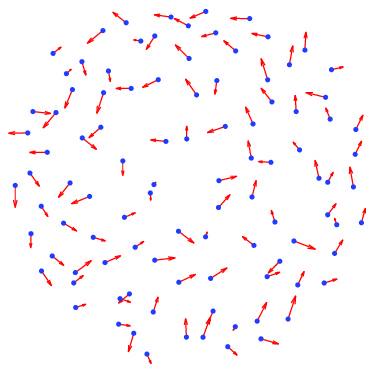}
\tabularnewline
\includegraphics[scale=0.2]{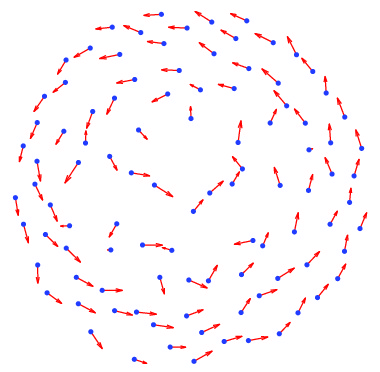}
&
\includegraphics[scale=0.2]{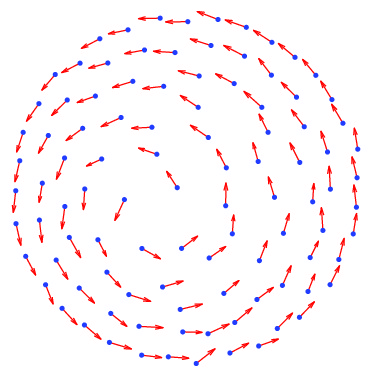}
&
\includegraphics[scale=0.2]{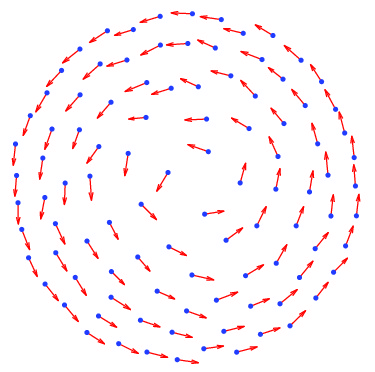}
&
\includegraphics[scale=0.2]{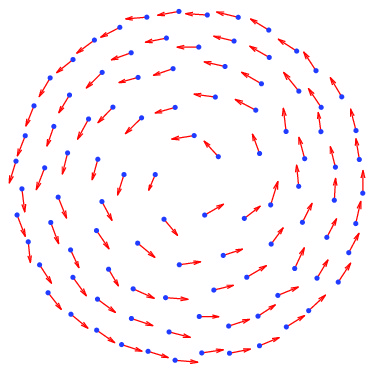}
\end{tabular}
\caption{System with $N=100$ agents, parameters $a=4$ and $b=0.001$ and $|u_0|=0.1$. The first row shows the instability of the flock ring solution while the second exhibits the convergence to a fat mill type solution.}
\label{tab:Flock2Mill}
\end{table}

These numerical tests show surprisingly that it is possible, with
a particular choice of the parameters, to obtain mill
configurations out of perturbations of initial flock solutions and
flock solutions out of perturbations of mill ring solutions. These
heteroclinic-kind solutions have not been previously reported. We
also remark that the parameter choice is connected to the number
of agents we are considering; changing $N$ means finding another
set of parameters for which the same switching behavior occurs.

\bibliographystyle{plain}
\bibliography{refs}

\end{document}